\documentclass[12pt]{amsart}
\usepackage{amsmath}	
\usepackage{amssymb}
\usepackage{caption}
\usepackage{tikz}
\usetikzlibrary{calc, decorations.markings}

\newtheorem{theorem}{Theorem}[section]
\newtheorem{lemma}[theorem]{Lemma}
\newtheorem{proposition}[theorem]{Proposition}
\newtheorem{cor}[theorem]{Corollary}

\theoremstyle{remark}

\theoremstyle{definition}

\numberwithin{equation}{section}

\newcommand{\bA}{\mathbb{A}}
\newcommand{\bC}{\mathbb{C}}
\newcommand{\bN}{\mathbb{N}}

\newcommand{\bQ}{\mathbb{Q}}
\newcommand{\bR}{\mathbb{R}}
\newcommand{\bZ}{\mathbb{Z}}

\newcommand{\cC}{{\mathcal{C}}}
\newcommand{\cD}{{\mathcal{D}}}
\newcommand{\cE}{{\mathcal{E}}}
\newcommand{\cG}{{\mathcal{G}}}
\newcommand{\cH}{{H}} 
\newcommand{\cK}{{\mathcal{K}}}

\newcommand{\cM}{{\mathcal{M}}}
\newcommand{\cO}{{\mathcal{O}}}

\newcommand{\cR}{{\mathcal{R}}}

\newcommand{\cU}{{\mathcal{U}}}

\newcommand{\card}{\mathrm{Card}}
\newcommand{\dt}{\/dt}
\newcommand{\dz}{\/dz}

\newcommand{\et}{\quad\text{and}\quad}

\newcommand{\Mat}{\mathrm{Mat}}

\newcommand{\Qbar}{\bar{\bQ}}
\newcommand{\ssi}{\quad\Longleftrightarrow\quad}

\newcommand{\tcC}{\tilde{\cC}}
\newcommand{\tgamma}{\tilde{\gamma}}
\newcommand{\tG}{\tilde{G}}

\newcommand{\ua}{\mathbf{a}}

\newcommand{\ue}{\mathbf{e}}

\newcommand{\un}{\mathbf{n}}

\newcommand{\uun}{\mathbf{1}}

\newcommand{\disp}{\displaystyle}

\renewcommand{\cG}{G}
\renewcommand\Re{\operatorname{Re}}
\renewcommand\Im{\operatorname{Im}}

\tikzset{->-/.style={decoration={
  markings,
  mark=at position #1 with {\arrow{>}}},postaction={decorate}},
  ->-/.default=0.5,
  }

\begin{document}

\baselineskip=16pt

\title[Approximation to values of the exponential function]
{Simultaneous approximation to values of the exponential function over the adeles}
\author{Damien Roy}

\thanks{Research partially supported by NSERC}

\begin{abstract}
We show that Hermite's approximations to values of the
exponential function at given algebraic numbers are
nearly optimal when considered from an adelic
perspective.  We achieve this by taking into account the ratio of
these values whenever they make sense in the various completions
(Archimedean or $p$-adic) of a number field containing these
algebraic numbers.
\end{abstract}

\maketitle

%%%%%%%%%%%%%%%%%%%%%%%%%%%%%%%%%%%%%%%%%%%%%%%%%%%%%%%%%%%%
%
%   Introduction
%
%%%%%%%%%%%%%%%%%%%%%%%%%%%%%%%%%%%%%%%%%%%%%%%%%%%%%%%%%%%%

\section{Introduction}
\label{sec:intro}

We know by Euler that the number $e$ admits a continued fraction
expansion consisting of intertwined arithmetic progressions
\[
 e=[2,(1,2n,1)_{n=1}^\infty] = [2,1,2,1,1,4,1,1,6,1,\dots].
\]
Euler, Sundman and Hurwitz also obtained similar expansions
for the numbers $e^{2/m}$ where $m$ is a non-zero integer
\cite[\S\S31-32]{Pe1929}.  Consequently, one may derive very
good measures of rational approximations to these numbres (see for
example the fully explicit results of Bundschuch \cite[Satz 2]{Bu1971},
is the case where $m$ is even).  This is the aspect that interests us
here.  We propose the following heuristic explanation:
the ratios $2/m$ with $m\in\bZ\setminus\{0\}$ are the only
non-zero rational numbers $z$ for which the usual power series
\begin{equation}
 \label{intro:eq:e^z}
 e^z=\sum_{k=0}^\infty \frac{z^k}{k!}
\end{equation}
converges only as a real number.  Indeed, let $p$ be a prime number
and let $\bC_p$ denote the completion of the algebraic closure
$\Qbar$ of $\bQ$ for the $p$-adic absolute value of $\bQ$ extended
to $\Qbar$, with $|p|_p=p^{-1}$.  We know that, for $z\in\bC_p$,
the series \eqref{intro:eq:e^z} converges in $\bC_p$ if and only
if $|z|_p<p^{-1/(p-1)}$. In particular, for a rational number $z$,
viewed as an element of $\bC_p$, this series converges if and only
if the numerator of $z$ is divisible by $p$ when $p\neq 2$, and
by $4$ when $p=2$.

This phenomenon also extends to algebraic numbers. Indeed,
let $K$ be a number field, namely an algebraic extension of $\bQ$
of finite degree. Then any absolute value on $K$ induces the same
topology on $K$ as an absolute value coming from an embedding from
$K$ into $\bC$ or into $\bC_p$ for a prime number $p$.  We say
that such embeddings define the same place $v$ of $K$ if they
induce the same absolute value on $K$ denoted $|\ |_v$.  We then
denote by $K_v$ the completion of $K$ for this absolute value.
When the place $v$ comes from an embedding of $K$ into $\bC$,
the place $v$ is called Archimedean and we write $v\mid \infty$.
Otherwise it is called ultrametric, and we write $v\mid p$ if
it comes from an embedding of $K$ into $\bC_p$.  When $\alpha\in K$,
the series for $e^\alpha$ converges in each Archimedian completion
of $K$ but only in a finite number of ultrametric completions.
In particular, when $K$ admits a single Archimedean place, which
happens when $K=\bQ$ or when $K$ is quadratic imaginary, then
it may occur that $e^\alpha$ has a meaning only for this place.
Then, we obtain the following estimate where $\cO_K$ denotes
the ring of integers of $K$.

\begin{proposition}
\label{intro:prop:imaginaire}
Let $K\subset \bC$ be the field $\bQ$ or a quadratic imaginary
extension of $\bQ$, and let $\alpha$ be a non-zero element of
$K$ such that $|\alpha|_v\ge p^{-1/(p-1)}$ for each prime
number $p$ and each place $v$ of $K$ with $v\mid p$.  Then,
for any $x,y\in\cO_K$ with $x\neq 0$, we have
\[
 |x|\,|xe^\alpha-y| \ge c(\log |x|)^{-2g-1}
\]
where $g$ stands for the number of places $v$ of $K$ with
$v\mid \infty$ or $|\alpha|_v\neq 1$, and where $c>0$ is a
constant depending only on $\alpha$ and $K$.
\end{proposition}

For example if $K=\bQ(\sqrt{-2})$, we may take
$\alpha=2(1\pm\sqrt{-2})/m$ where $m\in\cO_K\setminus\{0\}$.
If $K=\bQ(\sqrt{-23})$, we may take
$\alpha=(1\pm\sqrt{-23})/(2m)$ where $m\in\cO_K\setminus\{0\}$.
In some cases, $e^\alpha$ admits a generalized continued
fraction expansion similar to the one of $e$ (with partial
quotients in $\cO_K$) but we do not consider this question here.

More generally, let $\alpha_1,\dots,\alpha_s$ be distinct
elements of a number field $K\subset \bC$. Lindemann-Weierstrass
theorem \cite{We1885} tells us that their exponentials
$e^{\alpha_1},\dots,e^{\alpha_s}\in\bC$ are linearly
independent over $K$ and the classical proof,
in all variants (see \cite[Appendix]{Ma1976}),
is based on Hermite's approximations which we recall in the
next section.  Our goal is to show that these approximations
are nearly optimal in the context of geometry of numbers in
the adeles of $K$, when taking into account all places
$v$ of $K$ and all pairs of indices $i,j$ with $1\le i<j\le s$
for which the series for $e^{\alpha_i-\alpha_j}$
converges in $K_v$.  It is possible that this observation
reflects a much wider property of the values of the
exponential function.

For example the series for $e^3$ converges in $\bR$ and
in $\bQ_3$ but not in any $\bQ_p$ for a prime number $p\neq 3$.
Then our approach leads to the following result.

\begin{proposition}
\label{intro:prop:e3}
For any integer $n\ge 1$, we define a convex body $\cC_n$
of\/ $\bR^2$ and a lattice $\Lambda_n$ of\/ $\bR^2$ by
\begin{align*}
 \cC_{n}
  &=
 \left\{ (x,y)\in\bR^2 \,;
   \quad
   |x|\le \frac{(2n)!}{n!3^{n/2}}
   \,,\quad
   |xe^3-y|\le \Big(\frac{3}{2}\Big)^{2n}\frac{1}{n!3^{n/2}}
 \right\},\\
 \Lambda_{n}
  &=
 \left\{ (x,y)\in\bZ^2 \,;
   \quad
   |xe^3-y|_3\le 3^{-n}
 \right\}.
\end{align*}
For $i=1,2$, let $\lambda_i(\cC_n,\Lambda_n)$ denote
the $i$-th minimum of $\cC_n$
with respect to $\Lambda_n$, that is the smallest
$\lambda>0$ such that $\lambda\cC_n$ contains at least $i$
$\bQ$-linearly independent elements of $\Lambda_n$.
Then we have
\[
 (cn^2)^{-1}
  \le \lambda_1(\cC_n,\Lambda_n)
  \le \lambda_2(\cC_n,\Lambda_n)
  \le cn^2,
\]
for a constant $c>1$ that does not depend on $n$.
\end{proposition}

Using the fact that $3^n\bZ^2\subset\Lambda_n$, one deduces
that $\lambda_1(\cC_n,\bZ^2)\ge (c n^2 3^n)^{-1}$ for any
integer $n\ge 1$.  Consequently, for each $\epsilon>0$, there
exists a constant $c_\epsilon>0$ such that
\[
 |x|\,|xe^{3}-y|
  \ge c_\epsilon |x|^{-\epsilon}
\]
for all $(x,y)\in\bZ^2$ with $x\neq0$.  One may even derive
slightly sharper estimates (see \cite[Satz 1]{Bu1971}).  However,
numerical computations described in section \ref{sec:num} yield
\begin{equation}
 \label{intro:eq:loglog}
 |x|\,|xe^3-y|\ge (3\,\log|x|\,\log\log |x|)^{-1}
 \quad\text{if}\quad 4\le |x|\le 10^{500\,000}.
\end{equation}
More involved computations which we do not describe here even suggest
the existence of a real number $g>0$ such that
\[
 |x_1|\,|x_1e^3-x_2|\,|x_1e^3-x_2|_3 \ge (\log |x_1|)^{-g}
\]
for any $(x_1,x_2,x_3)\in\bZ^3$ with $|x_1|$ large enough.
Finally, an important result of Baker \cite{Ba1965} shows that
if $\alpha_2,\dots,\alpha_s\in\bQ$ are distinct non-zero
rational numbers then, for each $\epsilon>0$, there also
exists a constant $c_\epsilon>0$ such that
\[
 |x_1|\,|x_1e^{\alpha_2}-x_2|\cdots|x_1e^{\alpha_s}-x_s|
  \ge c_\epsilon |x_1|^{-\epsilon}
\]
for each $(x_1,\dots,x_s)\in\bZ^s$ with $|x_1|\neq 0$.  The
properties of Hermite's approximations suggest that the
right hand side $c_\epsilon |x_1|^{-\epsilon}$ in this inequality
could be remplaced by $(\log |x_1|)^{-g}$ for a constant $g>0$
depending only on $(\alpha_2,\dots,\alpha_s)$, when $|x_1|$ is
large enough.

In this paper, $\bN$ stands for the set of non-negative integers
and $\bN_+=\bN\setminus\{0\}$ for the set of positive integers.

\medskip
\noindent
\textbf{Acknowledgments:} I warmly thank
Michel Waldschmidt for numerous exchanges on these
questions. In particular, his course notes
\cite{Wa2008} were a source of inspiration.

%%%%%%%%%%%%%%%%%%%%%%%%%%%%%%%%%%%%%%%%%%%%%%%%%%%%%%%%%%%
%
%  Statement of the main result
%
%%%%%%%%%%%%%%%%%%%%%%%%%%%%%%%%%%%%%%%%%%%%%%%%%%%%%%%%%%%

\section{Statement of the main result}
\label{sec:resultat}

Let $K$ be a number field, let $\cO_K$ be its ring of
integers, let $d=[K:\bQ]$ be its degree
over $\bQ$, and let $s\in\bN_+$.
For any ultrametric place $v$ of $K$, we denote by
$\cO_v=\{x\in K_v\,;\,|x|_v\le 1\}$ the ring of integers of $K_v$
and by $d_v=[K_v:\bQ_p]$ the local degree of $K_v$, where $p$
stands for the prime number below $v$ (notation $v\mid p$),
namely the prime number $p$ for which $|\ |_v$ extends
the $p$-adic absolute value on $\bQ$.
Following McFeat \cite[\S2.2]{Mc1971}, we denote by
$\mu_v$ the Haar measure on $K_v$ normalized so that
$\mu_v(\cO_v)=1$.  For an Archimedean place (notation
$v\mid \infty$), we again denote by $d_v=[K_v:\bR]$ the
local degree of $K_v$, and define $\mu_v$ as the
Lebesgue measure on $K_v$ (this field is $\bR$ or $\bC$).
We denote by $r_1$ (resp.\ $r_2$) the number of
places $v\mid \infty$ with $d_v=1$ (resp.\ $d_v=2$),
so that $d=r_1+2r_2$.

The ring of adeles of $K$ is the product $K_\bA=\prod_v K_v$
running over all places $v$ of $K$, with the restricted
topology.  This is a locally compact ring that we equip with
the Haar measure $\mu$, product of the $\mu_v$.
We identify $K$ as a subfield of $K_\bA$ via the diagonal
embedding.  Then $K$ becomes a discrete subgroup of $K_\bA$
and, with the above normalization, we have
\[
 \mu(K_\bA/K)=2^{-r_2}|D(K)|^{1/2},
\]
where $D(K)$ stands for the discriminant of $K$.
By abuse of notation, we also write $\mu$ for the product
measure of $s$ copies of $\mu$ on $K_\bA^s$. Similarly,
for each place $v$ of $K$, we also write $\mu_v$ for the
product measure of $s$ copies of $\mu_v$ on $K_v^s$.
With our normalization of the absolute value on $K_v$,
if $T\colon K_v^s\to K_v^s$ is a $K_v$-linear map
and if $E$ is a measurable subset of $K_v^s$, the set
$T(E)$ is measurable with measure
$\mu_v(T(E))=|\det T|_v^{d_v}\mu_v(E)$.

%%%%%%%%%%%%%%%%%%%%
%
%  Minima of adelic convex bodies
%
%%%%%%%%%%%%%%%%%%%

\subsection{Minima of adelic convex bodies}
\label{subsec:minima}
An \emph{adelic convex body} of $K^s$ is a product
\[
  \cC=\prod_v \cC_v \subset K_\bA^s,
\]
indexed by all places $v$ of $K$, which satisfies
the following properties:
\begin{itemize}
\item[(i)] if $v\mid \infty$, then $\cC_v$ is a
  \emph{convex body} of $K_v^s$, namely a compact connected
  neighborhood of $0$ in $K_v^s$ such that $\alpha\,\cC_v=\cC_v$
  for any $\alpha\in K_v$ with $|\alpha|_v=1$;
\item[(ii)] if $v\nmid\infty$, then $\cC_v$ is a
  finite type (thus free) sub-$\cO_v$-module of $K_v^s$ of
  rank $s$;
\item[(iii)] $\cC_v=\cO_v^s$ for all but finitely many
  places $v$ of $K$ with $v\nmid \infty$.
\end{itemize}
Suppose that $\cC$ is such a product.  For each $i=1,\dots,s$,
we define its $i$-th minimum $\lambda_i(\cC)$ as the
smallest $\lambda>0$ for which the adelic convex body
\[
 \lambda\cC=\prod_{v\mid \infty}\lambda\cC_v \prod_{v\nmid\infty}\cC_v
\]
contains at least $i$ linearly independent elements of $K^s$
over $K$.  With this notation and our normalization of
measures, the adelic version of Minkowski's theorem
reads as follows.

\begin{theorem}[McFeat, Bombieri et Vaaler]
 \label{res:thm:MBV}
For any adelic convex body $\cC$ of $K^s$, we have
\[
  2^{sr_1}(s!)^{-d}
  \le
  \left(\lambda_1(\cC)\cdots\lambda_s(\cC)\right)^d\mu(\cC)
  \le
  2^{s(r_1+r_2)}|D(K)|^{s/2}.
\]
\end{theorem}

We refer the reader to \cite[Theorem 5]{Mc1971} and
\cite[Theorem 3]{BV1983} for the upper bound on the
product of the minima (see also the upper bound of Thunder in
\cite[Theorem 1 and Corollary]{Th2002}).  The lower bound
given here is taken from \cite[Theorem 6]{Mc1971};
it is slightly weaker than the one of \cite[Theorem 6]{BV1983}.

%%%%%%%%%%%%%%%%%%%%%%%%%
%
%  Hermite's approximations
%
%%%%%%%%%%%%%%%%%%%%%%%%%

\subsection{Hermite's approximations}
\label{subsec:approx:Hermite}
Let $\alpha_1,\dots,\alpha_s$ be distinct elements of $K$.
For each $s$-tuple $\un:=(n_1,\dots,n_s)\in\bN^s$, we define
polynomials of $K[z]$ by
\[
 f_\un(z)=(z-\alpha_1)^{n_1}\cdots(z-\alpha_s)^{n_s}
\et
 P_\un(z)=\sum_{k=0}^N f_\un^{(k)}(z)
\]
where
\[
 N=n_1+\cdots+n_s
\]
represents the degree of $f_\un$, and where
$f_\un^{(k)}$ denotes the $k$-th derivative
of $f_\un$ for each integer $k\ge 0$. We then form
the point
\[
 a_\un:=\big(P_\un(\alpha_1),\dots,P_\un(\alpha_s)\big)
   \in K^s.
\]
We call it the \emph{Hermite approximation
of order $\un$ for the $s$-tuple
$(\alpha_1,\dots,\alpha_s)$}.  Our goal is to give
a precise meaning to the term ``approximation'',
by working in the adeles of $K$.

We first recall some properties of these
points.  For simplicity, we start by assuming that
$K\subset \bC$.  We find
\begin{equation}
 \label{resultat:equadiff}
 \frac{d}{dz}\big( P_\un(z)e^{-z}\big)
  = \big( P_\un'(z)-P_\un(z) \big) e^{-z}
  =  -f_\un(z)e^{-z}.
\end{equation}
So, for any pair $i,j\in\{1,\dots,s\}$, we obtain
\begin{equation*}
 \label{resultat:integrale_finie}
 P_\un(\alpha_i)e^{-\alpha_i}-P_\un(\alpha_j)e^{-\alpha_j}
  = \int_{\alpha_i}^{\alpha_j} f_\un(z) e^{-t}\,\dz\,,
\end{equation*}
independently of the path of integration from $\alpha_i$
to $\alpha_j$ in $\bC$.  Upon integrating along the
line segment $[\alpha_i,\alpha_j]$ joining those two
points and observing that
\[
 \max_{z\in [\alpha_i,\alpha_j]} |f_\un(z)|
 \le
 R^N
 \quad\text{with}\quad
 R=\max_{1\le k,\ell\le s}|\alpha_k-\alpha_\ell|,
\]
we deduce that
\begin{equation*}
 \label{resultat:maj_diff}
 \left|
 P_\un(\alpha_i)e^{-\alpha_i}-P_\un(\alpha_j)e^{-\alpha_j}
 \right|
 \le
 c_1 R^N
\end{equation*}
for a constant $c_1>0$ that is independent of the choice of $i$,
$j$ and $\un$.  Similarly, for $i=1,\dots,s$, the formula
\eqref{resultat:equadiff} yields
\begin{equation*}
 \label{resultat:integrale_infinie}
 P_\un(\alpha_i)
  = \int_0^\infty f_\un(z+\alpha_i) e^{-z}\,\dz\,,
\end{equation*}
by integrating along $[0,\infty)\subset\bR$.
Since $|f_\un(t+\alpha_i)|\le (t+R)^N$ for all $t\ge 0$,
we deduce that
\begin{equation*}
 \label{resultat:maj_point}
 |P_\un(\alpha_i)|
  \le \int_0^\infty (t+R)^N e^{-t}\,\dt
  \le e^R\int_0^\infty t^N e^{-t}\,\dt
  = e^R N!\,.
\end{equation*}

More generally, let $v$ be any Archimedean place of $K$.
Put
\begin{equation}
 \label{resultat:eq:Rv}
 R_v=\max_{1\le k,\ell\le s}|\alpha_k-\alpha_\ell|_v
\end{equation}
and choose an embedding $\sigma\colon K\to\bC$
such that $|\alpha|_v=|\sigma(\alpha)|$ for all
$\alpha\in K$.  Then, for any pair of indices
$i,j\in\{1,\dots,s\}$, the above computations yield
\begin{align}
 \label{resultat:v:maj_diff}
 \left|
  P_\un(\alpha_i)e^{-\alpha_i}-P_\un(\alpha_j)e^{-\alpha_j}
 \right|_v
 &=
 \left|
  \int_{\sigma(\alpha_i)}^{\sigma(\alpha_j)} f^\sigma_\un(z) e^{-z}\dz
 \right|
 \le c_v R_v^N,\\
 \label{resultat:v:maj_point}
 \left| P_\un(\alpha_i) \right|_v
 &\le
  e^{R_v} N!\,,
\end{align}
where $f^\sigma_\un$ denotes the image of $f_\un$ under
the ring homomorphism from $K[z]$ to $\bC[z]$ which fixes $z$
and coincides with $\sigma$ on $K$, and where $c_v>0$
depends only on $v$ and $\alpha_1,\dots,\alpha_s$.  Thus,
$a_\un$ is a projective approximation to
$(e^{\alpha_1},\dots,e^{\alpha_s})$
at each Archimedean place of $K$.

In this paper, we establish an upper bound for the
integral in \eqref{resultat:v:maj_diff} which is
sharper than $c_vR_v^N$ for each Archimedean place $v$
of $K$.  We also provide analogs of
\eqref{resultat:v:maj_diff} and of \eqref{resultat:v:maj_point}
for the ultrametric places $v$ of $K$ whenever their
left hand side makes sense in $K_v$.
More precisely, as $e^{\alpha_j-\alpha_i}$ could
make sense in $K_v$ without $e^{\alpha_i}$ and
$e^{\alpha_j}$ making sense, we consider instead
the quantities
$|P_\un(\alpha_i)e^{\alpha_j-\alpha_i}-P_\un(\alpha_j)|_v$.
Here again, we will need sharp estimates while
usually the ultrametric places are treated in an
expeditious manner.  In general, one chooses a common
denominator $b$ of $\alpha_1,\dots,\alpha_s$, that is an
integer $b\ge 1$ such that
$b\alpha_1,\dots,b\alpha_s\in \cO_K$.  Then the polynomial
$g(z):=b^Nf(z/b)$ has coefficients in $\cO_K$ and, for each
$i=1,\dots,s$, we find
\begin{equation*}
 \label{resultat:denom}
 \frac{b^N}{(n_i)!}P_\un(\alpha_i)
 = \sum_{k=n_i}^N \frac{b^N}{(n_i)!} f^{(k)}(\alpha_i)
 = \sum_{k=n_i}^N \frac{b^k k!}{(n_i)!} \cdot \frac{g^{(k)}(b\alpha_i)}{k!}
 \in \cO_K\,.
\end{equation*}
For example, if $n_1=\cdots=n_s=n$, this implies
that $(b^N/n!)a_\un\in\cO_K^s$.

The above estimates are key-ingredients in the classical
proof of the Lindemann-Weiertrass theorem asserting that
$e^{\alpha_1},\dots,e^{\alpha_s}$ are linearly independent
over $K$.  However, two more ingredients are missing.  The first
one is a reduction step of Weierstrass which is explained
in \cite[Appendix, \S3]{Ma1976}
(see also \cite[Chapter 1, \S3]{Ba1975}).  The second one is
the existence of families of $s$ linearly independent
approximations over $K$.  Hermite himself noticed this problem
and solved it in order to prove the transcendence of $e$.  We
will use here the following remarkable result of Mahler.

\begin{theorem}[Mahler]
\label{res:thm:Mahler}
Suppose that $\un=(n_1,\dots,n_s)\in\bN_+^s$ has positive
coordinates.  Let $\ue_1=(1,0,\dots,0),\dots,
\ue_s=(0,\dots,0,1)$ denote the canonical basis elements
of\/ $\bZ^s$.  Then, we have
\begin{equation}
 \label{resultat:det:Mahler}
 \Delta_\un
 :=
 \det(a_{\un-\ue_1},\dots,a_{\un-\ue_s})
 =
 \prod_{i=1}^s
     \left(
     (n_i-1)!\prod_{k\neq i} (\alpha_i-\alpha_k)^{n_k}
     \right)
 \neq 0.
\end{equation}
\end{theorem}

The proof of Mahler is clever.  It is presented in
\cite[\S8]{Ma1932} and again in \cite[Appendix, \S16]{Ma1976}.
In the case where
$n_1=\dots=n_s$, the result is due to Hermite
\cite{He1873}.  Hermite's proof is different.  It is
based on the recurrence relations satisfied
by the points points $\ua_\un$ which we generalize in
Appendix \ref{sec:rel}.

%%%%%%%%%%%%%%%%%%%
%
%  Statement of the main result
%
%%%%%%%%%%%%%%%%%%%

\subsection{Statement of the main result}
\label{subsec:res}
With the above notation, let $E$ be the finite set
consisting of all Archimedean places of $K$
together with the ultrametric places $v$ of $K$ such that
$|\alpha_i-\alpha_j|_v\neq 1$ for at least one pair of indices
$i,j\in\{1,\dots,s\}$ with $i\neq j$.  For each $s$-tuple
$\un=(n_1,\dots,n_s)\in\bN_+^s$, we let $N$ denote its sum
and we define an adelic convex body $\cC_\un=\prod_v\cC_{\un,v}$
of $K^s$ as follows.
\begin{itemize}
\item[(i)] If $v\,|\,\infty$ is the place attached to an
 embedding $\sigma\colon K\hookrightarrow\bC$, we define
 $R_v$ by \eqref{resultat:eq:Rv}. Then $\cC_{\un,v}$ is the
 set of points $(x_1,\dots,x_s)\in K_v^s$ which satisfy
\begin{equation}
 \label{res:eq:C_arch}
 |x_i|_v\le e^{R_v} (N-1)!
 \et
 |x_ie^{\alpha_j-\alpha_i}-x_j|_v
 \le \max_{1\le k\le s}
     \left|
       \int_{\sigma(\alpha_i)}^{\sigma(\alpha_j)}
          f_{\un-\ue_k}^\sigma(z)e^{\sigma(\alpha_j)-z}\,\dz
     \right|
\end{equation}
for each pair of indices $i,j\in\{1,\dots,s\}$ with $i\neq j$.
\smallskip
\item[(ii)]
If $v\in E$ and if $v\,|\,p$ for a prime number $p$,
then $\cC_{\un,v}$ is the set of points %consists of all points
$(x_1,\dots,x_s)$ in $K_v^s$ which satisfy
\begin{equation}
 \label{res:eq:ultra1}
 |x_i|_v
  \le p^3 N
      \prod_{1\le k\le s}
      \max\big\{|\alpha_i-\alpha_k|_v,\,p^{-1/(p-1)}\big\}^{n_k}
\end{equation}
for $i=1,\dots,s$, as well as
\begin{equation}
 \label{res:eq:ultra2}
 |x_ie^{\alpha_j-\alpha_i}-x_j|_v
 \le p^3 N
      \prod_{1\le k\le s}
      \max\big\{|\alpha_i-\alpha_k|_v,|\alpha_j-\alpha_k|_v\big\}^{n_k}.
\end{equation}
for each pair of integers $i,j\in\{1,\dots,s\}$ such that
$0 < |\alpha_j-\alpha_i|_v < p^{-1/(p-1)}$.
\smallskip
\item[(iii)] Finally, if $v\notin E$, then $\cC_{\un,v}$ is the
set of points $(x_1,\dots,x_s)\in K_v^s$ satisfying
\[
 |x_i|_v \le |(n_i-1)!|_v
\]
for $i=1,\dots,s$.
\end{itemize}
The crucial feature of these adelic convex bodies
$\cC_\un$ is that the linear forms which define them
involve only the complex or $p$-adic values of the
exponential function at the points $\alpha_i$ or
$\alpha_j-\alpha_i$.  In view of the estimates in
\S\ref{subsec:approx:Hermite}, their component $\cC_{\un,v}$
contains the points $a_{\un-\ue_1},\dots,a_{\un-\ue_s}$
for each Archimedean place $v$ of $K$.  We will show in the
next section that this holds in fact for all places of $K$,
yielding the first assertion in the following result.

\begin{theorem}
\label{res:principal}
Let $\un=(n_1,\dots,n_s)\in\bN_+^s$. Then the adelic convex
body $\cC_\un$ contains the points
$a_{\un-\ue_1},\dots,a_{\un-\ue_s}$.  Moreover, upon
setting $N=n_1+\cdots+n_s$, we have the following volume
estimates.
\begin{itemize}
\item[(i)] If\/ $v\mid\infty$, then
 \[
   (s!)^{-1} |\Delta_\un|_v
    \le \mu_v(\cC_{\un,v})^{1/d_v}
    \le c_v N^{2s-2} |\Delta_\un|_v
 \]
 for a constant $c_v>0$ depending only on
 $\alpha_1,\dots,\alpha_s$ and $v$.
\smallskip
\item[(ii)] If\/ $v\in E$ and if $v\mid p$ for a prime number $p$, then
 \[
   |\Delta_\un|_v\le \mu_v(\cC_{\un,v})^{1/d_v} \le (p^3N)^s |\Delta_\un|_v.
 \]
\item[(iii)] If\/ $v\notin E$, then $\mu_v(\cC_{\un,v})^{1/d_v}=|\Delta_\un|_v$.
\end{itemize}
\end{theorem}

Note that, for each place $v$ of $K$, these estimates enclose
the volume of $\cC_{\un,v}$ between limits whose ratio
is a polynomial in $N$ while these limits themselves grow
like $|\Delta_\un|_v$, that is roughly like an
exponential in $N$ if $v\nmid\infty$ or like $N!$
if $v\mid \infty$.  When $v\mid \infty$, we give an explicit
value for the constant $c_v$ in Theorem \ref{volarch:thm}.

The lower bounds for $\mu_v(\cC_{\un,v})$ follow
easily from the definition of $\Delta_\un$ as a determinant
in \eqref{resultat:det:Mahler},
if we take for granted the fact that $\cC_{\un,v}$ contains the
points $\ua_{\un-\ue_i}$ for $i=1,\dots,s$.  Indeed, let
$T\colon K_v^s\to K_v^s$ be the $K_v$-linear map defined by
\[
 T(x_1,\dots,x_s)=x_1\ua_{\un-\ue_1}+\cdots+x_s\ua_{\un-\ue_s}
\]
for each $(x_1,\dots,x_s)\in K_v^s$.  Then
$\cC_{\un,v}$ contains $T(\cE_v)$ where $\cE_v$ is given by
\begin{align*}
  \cE_v &=\{(x_1,\dots,x_s)\in K_v^s\,;\,|x_1|_v+\cdots+|x_s|_v\le 1\}
  &&\text{if $v\mid \infty$,}\\
  \cE_v &=\cO_v^s
  &&\text{if $v\nmid\infty$.}
\end{align*}
As $|\det T|_v=|\Delta_\un|_v$, we have $\mu_v(T(\cE_v))
=|\Delta_\un|_v^{d_v}\mu_v(\cE_v)$.  If $v\mid\infty$,
we also have $\mu_v(\cE_v)\ge (s!)^{-d_v}$, thus $\mu_v(\cC_{\un,v})^{1/d_v}
\ge (s!)^{-1}|\Delta_\un|_v$.  If $v\nmid\infty$, we simply have
$\mu_v(\cE_v)=1$, thus $\mu_v(\cC_{\un,v})^{1/d_v} \ge |\Delta_\un|_v$.

Our main contribution therefore lies in the upper bounds for the volume
of the components $\cC_{\un,v}$, and we explain our strategy below.  These
upper bounds in turn yield an upper bound for the volume of $\cC_\un$
from which we derive the following conclusion thanks to
the adelic Minkowski theorem.

\begin{cor}
\label{res:cor}
In the notation of Theorem \ref{res:principal}, we have
\[
 cN^{-g}\le \lambda_1(\cC_\un)\le\cdots\le\lambda_s(\cC_\un)\le 1
 \quad
 \text{where}\quad
 g=s-2+s\sum_{v\in E} \frac{d_v}{d},
\]
and where $c>0$ is a constant depending only on $\alpha_1,\dots,\alpha_s$.
\end{cor}

\begin{proof}
Since $\prod_v |\Delta_\un|_v^{d_v} = 1$ and since $E$ contains all
Archimedean places of $K$, we find
\[
 \mu(\cC_n)
  = \prod_v \mu_v(\cC_{\un,v})
  \le \prod_{v\mid \infty} \Big(c_v^{d_v}N^{(2s-2)d_v}\Big)
      \prod_{v\in E,\,v\mid p} (p^3N)^{sd_v}
  = c_1^d N^{gd}
\]
where $c_1>0$ is independent of $\un$.
Since $\cC_\un$ contains the points $a_{\un-\ue_1},\dots,a_{\un-\ue_s}$
of $K^s$ and since, by Theorem \ref{res:thm:Mahler}, these points
are linearly independent over $K$, we also have
\[
  \lambda_1(\cC_\un)\le\cdots\le\lambda_s(\cC_\un)\le 1.
\]
Thus, by Theorem \ref{res:thm:MBV}, we obtain
\[
 (s!)^{-1}
  \le \lambda_1(\cC_\un)\cdots\lambda_s(\cC_\un)\mu(\cC_\un)^{1/d}
  \le \lambda_1(\cC_\un)c_1N^g,
\]
so $\lambda_1(\cC_\un)\ge cN^{-g}$ with $c=1/(c_1s!)$.
\end{proof}

The proof of Theorem \ref{res:principal} uses general
results on univariate polynomials $f(z)\in\bC[z]$
which we could not find in the
literature.  Suppose that $f$ has degree $N\ge 1$.
Let $A$ be its set of roots in $\bC$ and let $B$ be the set
of roots of its derivative $f'$ which do not belong to $A$.
In section \ref{sec:descente}, we consider the paths
of steepest descent for $|f|$ starting from an arbitrary
point $\beta$ of $\bC$.  These paths necessarily end in an
element of $A$. We show that they are contained in the
convex hull of $A\cup\{\beta\}$, with length at most
$\pi RN$ where $R$ is the radius of any disk containing
$A\cup\{\beta\}$.  In section \ref{sec:arch}, for each
$\beta\in B$, we denote by $m(\beta)$ the multiplicity of
$\beta$ as a root of $f'$ and, starting from $\beta$,
we choose $m(\beta)+1$ paths of steepest descent for $|f|$
which are locally distinct in a neighborhood of $\beta$.
These paths draw a graph on $A\cup B$ and we show that
this graph is in fact a tree.  We extract from it a sub-graph
$G$ on $A$ which is also a tree with edges indexed by $B$.
Then, for each edge of $G$ with end points $\alpha,\alpha'
\in A$, indexed by $\beta\in B$, we obtain a path
joining $\alpha$ to $\alpha'$ passing through
$\beta$, with length at most $2\pi RN$, along which $|f|$ is
maximal at the point $\beta$.

For the proof of Theorem \ref{res:principal} (i), we may
assume that the given place $v\mid \infty$ comes from
an inclusion $K\subset \bC$.  We then apply the above
construction, choosing $f$ to be the gcd of the
polynomials $f_{\un-\ue_1},\dots,f_{\un-\ue_s}$.  If the
coordinates of $\un\in\bN_+^s$ are all $\ge 2$, we thus
obtain a tree $G$ on $A=\{\alpha_1,\dots,\alpha_s\}$.
Then, for each edge of $G$ with end points $\alpha_i,\alpha_j$,
we bound from above the integrals in \eqref{res:eq:C_arch}
as a function of $|f(\beta)|$ where $\beta\notin A$ is
the corresponding root of $f'$.  From this, we deduce in
section \ref{sec:volarch} an upper bound for the volume of the
convex body $\cC_{\un,v}$ in terms of the product of the
values $|f(\beta)|^{m(\beta)}$ with $\beta\in B$, this being
the Chudnovsky semi-resultant of $f$ and $f'$.  The upper bound for
$\mu_v(\cC_{\un,v})$ then follows thanks to the computation
of this semi-resultant in section \ref{sec:semi-resultant}.
The general case where at least one coordinate of $\un$ is
equal to $1$ requires a slight adjustment.

The treatment of the ultrametric places $v\nmid\infty$ is
simpler.  In section \ref{sec:estultra}, we show that
$\cC_{\un,v}$ contains the points
$\ua_{\un-\ue_1},\dots,\ua_{\un-\ue_s}$.  Afterwards,
in section \ref{sec:ultra}, we construct a rooted forest on
$\{\alpha_1,\dots,\alpha_s\}$ associated with the place $v$.
This allows us to select $s$ inequalities among
\eqref{res:eq:ultra1} and \eqref{res:eq:ultra2} and to
deduce from them the required upper bound on the volume
of $\cC_{\un,v}$ in section \ref{sec:volultra}.
The relevant notions from graph theory are recalled in
section \ref{sec:graphes}.

In section \ref{sec:dexp}, we restrict to ``diagonal''
approximations to two exponentials, namely to the case
$s=2$ and $n_1=n_2$.  In this situation, we provide a refined form
of our main result whose proof relies only on the estimates from
sections \ref{subsec:approx:Hermite} and \ref{sec:estultra}.
We then use it to prove Propositions
\ref{intro:prop:imaginaire} and \ref{intro:prop:e3} from the introduction.

We conclude in section \ref{sec:num} by explaining how
Hermite's recurrence formulas recalled in Appendix
\ref{sec:rel} can be used to compute efficiently
the partial quotients in the continued fraction expansion
of $e^3$.  This in turn permits to validate the inequalities
\eqref{intro:eq:loglog} in less than two hours of
computation on a small desk computer.

%%%%%%%%%%%%%%%%%%%%%%%%%%%%%%%%%%%%%%%%%%%%%%%%%%%%%
%
%  Ultrametric estimates
%
%%%%%%%%%%%%%%%%%%%%%%%%%%%%%%%%%%%%%%%%%%%%%%%%%%%%%

\section{Ultrametric estimates}
\label{sec:estultra}

Let $v$ be a place of $K$ above a prime number $p$.
In this section, we complete
the proof of the first assertion in Theorem \ref{res:principal}
by showing that the component $\cC_{\un,v}$ of $\cC_\un$
contains the points $\ua_{\un-\ue_1},\dots,\ua_{\un-\ue_s}$
for each $\un\in\bN_+^s$. To this end, we use the following
notation and results.

For each $a\in\bC_p$ and each $r>0$, we denote by
\[
 B(a,r)=\{z\in\bC_p\,;\, |z-a|_p\le r\}
\]
the closed disk of $\bC_p$ with center $a$ and radius $r$
(both closed and open in $\bC_p$).  For such a disk
$B=B(a,r)$ and for any analytic function $g\colon B\to \bC_p$,
we define
\[
 |g|_B=\sup\{|g(z)|_p\,;\, z\in B\}.
\]
This quantity can also be computed from the Taylor series
expansion of $g$ around the point $a$ via the formula
\[
 |g|_B=\sup_{k\in\bN} \left|\frac{g^{(k)}(a)}{k!}\right|_p r^k,
\]
which yields the $p$-adic form of Cauchy's inequalities
\[
 |g^{(k)}(a)|_p \le |k!|_p r^{-k} |g|_B
 \quad
 (k\in\bN)
\]
(see \cite[\S1.5]{Ro1978}).  For the computations, we also use
the estimates
\begin{equation}
 \label{estultra:eq:delta}
 \delta^k\le |k!|_p \le k\delta^{k-p} \le p^2k\delta^k
 \quad
 (k\in \bN),
 \quad
 \text{where}
 \quad
 \delta=p^{-1/(p-1)},
\end{equation}
which follow from the formula $|k!|_p=p^{-m}$ where
$m=\sum_{\ell=1}^\infty \lfloor k/p^\ell \rfloor$.

\begin{lemma}
 \label{estultra:lemme1}
Let $\un=(n_1,\dots,n_s)\in\bN^s$, let $N=n_1+\cdots+n_s$, and
let $i,j\in\{1,\dots,s\}$. Then, we have
\begin{equation}
 \label{estultra:lemme1:eq1}
 |P_\un(\alpha_i)|_v
  \le p^2 N \prod_{k=1}^s \max\{|\alpha_i-\alpha_k|_v,\,\delta\}^{n_k}.
\end{equation}
If $|\alpha_i-\alpha_k|_v\le 1$ for $k=1,\dots,s$,
we also have
\begin{equation}
 \label{estultra:lemme1:eq1bis}
 |P_\un(\alpha_i)|_v
  \le |n_i!|_v.
\end{equation}
Finally, if $\rho=|\alpha_i-\alpha_j|_v$ satisfies $0<\rho<\delta$,
we have
\begin{equation}
 \label{estultra:lemme1:eq2}
 |P_\un(\alpha_i)e^{\alpha_j-\alpha_i}-P_\un(\alpha_j)|_v
  \le \frac{\rho}{\delta} p^2  N \prod_{k=1}^s
      \max\{|\alpha_i-\alpha_k|_v,\,|\alpha_j-\alpha_k|_v\}^{n_k}.
\end{equation}
\end{lemma}

\begin{proof}
To simplify, we may assume that $K\subset\bC_p$
and that $|\alpha|_v=|\alpha|_p$ for each $\alpha\in K$.  Then,
the polynomial $f_\un(z)\in K[z]$ can be viewed as an analytic
function $f_\un\colon\bC_p\to\bC_p$.  To estimate $|P_\un(\alpha_i)|_v
= |P_\un(\alpha_i)|_p$, we set
\[
 B=B(\alpha_i,\delta)
 \et
 M=|f_\un|_B.
\]
For $k=0,1,\dots,N$, Cauchy's inequalities together with
\eqref{estultra:eq:delta} yield
\[
 |f_\un^{(k)}(\alpha_i)|_p
  \le |k!|_p \delta^{-k} M
  \le p^2 k M
  \le p^2 N M,
\]
thus
\[
 |P_\un(\alpha_i)|_v
  = \left| \sum_{k=0}^N f_\un^{(k)}(\alpha_i)\right|_p
  \le p^2 N M.
\]
This proves \eqref{estultra:lemme1:eq1} since
\[
 M\le \prod_{k=1}^s \sup\{|z-\alpha_k|_p\,;\,z\in B\}^{n_k}
    = \prod_{k=1}^s \max\{|\alpha_i-\alpha_k|_v,\,\delta\}^{n_k}.
\]
If $|\alpha_i-\alpha_k|_v\le 1$ for each $k$,
a similar computation yields $|f_\un|_B\le 1$ with
$B=B(\alpha_i,1)$.  Then Cauchy's inequalities give
$|f_\un^{(k)}(\alpha_i)|_p\le |k!|_p$ for each $k\in\bN$.
Since we have $f_\un^{(k)}(\alpha_i)=0$ for $k=0,\dots,n_i-1$,
we deduce that $|f_\un^{(k)}(\alpha_i)|_v\le |n_i!|_v$
for each $k\in\bN$ and the upper bound
\eqref{estultra:lemme1:eq1bis} follows.

Suppose now that $0< \rho = |\alpha_i-\alpha_j|_p < \delta$.  To
prove \eqref{estultra:lemme1:eq2}, we use instead
\[
 B = B(\alpha_j,\rho)
 \et
 M=|f_\un|_B.
\]
Since $\rho<\delta$, the function $g\colon B\to \bC_p$ given by
\[
 g(z)=P_\un(z)e^{\alpha_j-z}-P_\un(\alpha_j)
 \quad (z\in B)
\]
is analytic with $ g(\alpha_j)=0$ and
\begin{equation}
 \label{estultra:lemme1:eq3}
 g'(z)=-f_\un(z)e^{\alpha_j-z}
 \quad (z\in B).
\end{equation}
For each integer $\ell=0,1,\dots,N$, we have
\[
 |f_\un^{(\ell)}(\alpha_j)|_p
  \le |\ell!|_p \rho^{-\ell} M
  \le p^2 \ell (\delta/\rho)^\ell M
  \le p^2 N (\delta/\rho)^\ell M.
\]
Since $\disp f_\un^{(\ell)}=0$ for $\ell>N$, this remains valid
for each $\ell\in\bN$.  Then, by \eqref{estultra:lemme1:eq3},
Leibniz formula for the derivative of a product yields,
for each integer $k\ge 1$,
\[
 |g^{(k)}(\alpha_j)|_p
  \le \max_{0\le \ell <k} |f_\un^{(\ell)}(\alpha_j)|_p
  \le p^2 N (\delta/\rho)^{k-1} M.
\]
Since $\alpha_i\in B$ and $g(\alpha_j)=0$, we deduce that
\[
 |P_\un(\alpha_i)e^{\alpha_j-\alpha_i}-P_\un(\alpha_j)|_v
  =|g(\alpha_i)|_p
  \le |g|_B
   =  \sup_{k\ge 1} \left|\frac{g^{(k)}(\alpha_j)}{k!}\right|_p
       \rho^k
  \le p^2 N (\rho/\delta) M.
\]
The upper bound \eqref{estultra:lemme1:eq2} follows since
\[
 M\le \prod_{k=1}^s \sup\{|z-\alpha_k|_p\,;\,z\in B\}^{n_k}
    = \prod_{k=1}^s \max\{|\alpha_i-\alpha_k|_v,\,|\alpha_j-\alpha_k|_v\}^{n_k}.
\qedhere
\]
\end{proof}

\begin{theorem}
Let $\un=(n_1,\dots,n_s)\in\bN_+^s$.  Then
the subset %the sub-$\cO_v$-module
$\cC_{\un,v}$
of $K_v^s$ defined in Section \ref{subsec:res}
contains the points
$\ua_{\un-\ue_1},\dots,\ua_{\un-\ue_s}$.
\end{theorem}

\begin{proof}
Fix an integer $\ell\in\{1,\dots,s\}$ and put
$P=P_{\un-\ue_\ell}$.  To show that $\cC_{\un,v}$ contains
the point $\ua_{\un-\ue_\ell}= (P(\alpha_1),\dots,P(\alpha_s))$,
we fix arbitrary $i,j\in\{1,\dots,s\}$.  Since
$\max\{|\alpha_i-\alpha_\ell|_v,\,\delta\}\ge \delta\ge 1/p$,
the inequality \eqref{estultra:lemme1:eq1} of
Lemma \ref{estultra:lemme1} provides
\[
 \big|P(\alpha_i)\big|_v
 \le p^2(N-1) \frac{1}{\delta} \prod_{k=1}^s
     \max\{|\alpha_i-\alpha_k|_v,\,\delta\}^{n_k}
 \le p^3 N \prod_{k=1}^s \max\{|\alpha_i-\alpha_k|_v,\,\delta\}^{n_k}.
\]
If $|\alpha_i-\alpha_k|_v=1$ for each $k=1,\dots,s$ with $k\neq i$,
the inequality \eqref{estultra:lemme1:eq1bis} of the same lemma
also provides
\[
 \big|P(\alpha_i)\big|_v\le |(n_i-1)!|_v.
\]
Finally, if $\rho=|\alpha_j-\alpha_i|_v$ satisfies $0<\rho<\delta$,
then,
since $\max\{|\alpha_i-\alpha_\ell|_v,\,|\alpha_j-\alpha_\ell|_v\}
\ge \rho$, the inequality \eqref{estultra:lemme1:eq2} yields
\begin{align*}
 |P(\alpha_i)e^{\alpha_j-\alpha_i}-P(\alpha_j)|_v
  &\le \frac{\rho}{\delta} p^2  (N-1)\cdot \frac{1}{\rho} \prod_{k=1}^s
      \max\{|\alpha_i-\alpha_k|_v,\,|\alpha_j-\alpha_k|_v\}^{n_k}\\
  &\le  p^3 N \prod_{k=1}^s
      \max\{|\alpha_i-\alpha_k|_v,\,|\alpha_j-\alpha_k|_v\}^{n_k}.
\qedhere
\end{align*}
\end{proof}

%%%%%%%%%%%%%%%%%%%%%%%%%%%%%%%%%%%%%%%%%%%%%%%%%%%
%
%  Preliminaries of graph theory
%
%%%%%%%%%%%%%%%%%%%%%%%%%%%%%%%%%%%%%%%%%%%%%%%%%%

\section{Preliminaries of graph theory}
\label{sec:graphes}

A \emph{graph} $\cG$ is a pair of finite sets $(V,E)$
where $E$ consists of subsets of $V$ with two elements.
The elements of $V$ are called the \emph{vertices}
of $\cG$ and those of $E$ the \emph{edges} of $\cG$
in agreement with the usual graphic representation.

Let $\cG=(V,E)$ be a graph.  An \emph{elementary chain}
in $\cG$ is a sequence $(\alpha_1,\dots,\alpha_m)$ of
$m\ge 2$ distinct elements of $V$ such that
$\{\alpha_i,\alpha_{i+1}\}\in E$ for $i=1,\dots,m-1$.
We say that $\cG$ is \emph{connected} if, for each pair
of distinct elements $\alpha,\beta$ of $V$, there exists
at least one elementary chain $(\alpha_1,\dots,\alpha_m)$
in $\cG$ with $\alpha_1=\alpha$ and $\alpha_m=\beta$.
We say that $\cG$ is a \emph{tree} if there exists exactly
one such chain for each choice of $\alpha,\beta\in V$
with $\alpha\neq \beta$.
When $\cG$ is connected, we have $|V|\le |E|+1$ with equality
if and only if $\cG$ is a tree.

In general, for a graph $\cG=(V,E)$, there exists one and
only one choice of integer $r\ge 1$ and partitions
$V = V_1\cup\cdots\cup V_r$ and $E = E_1\cup\cdots\cup E_r$
of $V$ and $E$ into $r$ disjoint subsets such that
$\cG_i=(V_i,E_i)$ is a connected graph for $i=1,\dots,r$.
We say that $\cG_1,\dots,\cG_r$ are the \emph{connected
components} of $\cG$.  If these are trees, we say that
$\cG$ is a \emph{forest}.
% Dans ce cas, on a $|V|=|E|+r$.
When $\cG$ admits $r$ connected components, we have
$|V|\le |E|+r$ with equality if and only if $\cG$
is a forest.

A \emph{rooted forest} is a triple $\cG=(R,V,E)$
where $(V,E)$ is a forest and where $R$ is a subset
of $V$ containing exactly one vertex from each
connected component of $(V,E)$.
We say that $R$ is the set of \emph{roots}
of $\cG$.  Then, for each $\beta\in V\setminus R$,
there is a unique elementary chain
$(\alpha_1,\dots,\alpha_m)$ with $\alpha_1\in R$ and
$\alpha_m=\beta$.  So we obtain a partial ordering on
$V$ by defining $\alpha<\beta$ if $\beta\notin R\cup\{\alpha\}$
and if the elementary chain which links $\beta$ to an
element of $R$ contains $\alpha$.  In particular, any edge
$\{\alpha,\beta\}\in E$ can be ordered so that $\alpha<\beta$.
The resulting pairs $(\alpha,\beta)$ are called the
\emph{directed edges} of $G$. For fixed $\alpha\in V$,
we say that $D_\cG(\alpha)=\{\beta\in V\,;\, \alpha<\beta\}$
is the set of \emph{descendants} of $\alpha$. The set $S_\cG(\alpha)$
of minimal elements of $D_\cG(\alpha)$ is called the set of
\emph{successors} of $\alpha$.  Note that the pairs
$(\alpha,\beta)\in V\times V$ with $\beta\in S_\cG(\alpha)$
are exactly the directed edges of $G$.
Moreover, any $\beta\in V\setminus R$ is the successor of a
unique $\alpha\in V$.  This allows us to formulate the
following result.

\begin{proposition}
\label{graphes:prop}
Let $\cG=(R,V,E)$ be a rooted tree, let $K$ be a field,
let $(x_\alpha)_{\alpha\in V}$ be a family of indeterminates
over $K$ indexed by $V$, and let $\varphi\colon E\to K$
be a function.  For each $\beta\in V$, we define
\[
 L_\beta
 = \begin{cases}
   x_\beta
     &\text{if $\beta\in R$,}\\
   x_\beta-\varphi(\{\alpha,\beta\})x_\alpha
     &\text{if $\beta\in S_\cG(\alpha)$ with $\alpha\in V$.}
   \end{cases}
\]
Then, upon extending the partial ordering on $V$ to a total
ordering, the matrix of the linear forms $(L_\beta)_{\beta\in V}$
with respect to the basis $(x_\alpha)_{\alpha\in V}$ is lower
triangular with $1$ everywhere on the diagonal.
\end{proposition}

%%%%%%%%%%%%%%%%%%%%%%%%%%%%%%%%%%%%%%%%%%%%%%%%%%%%%%%%%
%
%  Chemins de descente maximale
%
%%%%%%%%%%%%%%%%%%%%%%%%%%%%%%%%%%%%%%%%%%%%%%%%%%%%%%%%%

%\newpage
\section{Paths of steepest ascent}
\label{sec:descente}

In this section, we fix a non-constant monic polynomial
$f(z)\in\bC[z]$, a compact convex subset $\cK$ of $\bC$
containing all the roots of $f$, and a closed disk $D$
of $\bC$ containing $\cK$.  We denote by $N$ the degree
of $f$, and by $R$ the radius of $D$.  The main goal of
this section is to prove the following result.

\begin{theorem}
\label{descente:thm}
Let $\beta\in \cK$. There exists a root $\alpha$ of $f$
and a path $\gamma\colon[0,1]\to\bC$ linking
$\gamma(0)=\alpha$ to $\gamma(1)=\beta$, such that
$f(\gamma(t))=tf(\beta)$ for each $t\in[0,1]$.  The image
of such a path is contained in $\cK$, with length at most
$\pi RN$.
\end{theorem}

By a \emph{path} we mean here a continuous piecewise
differentiable map $\gamma\colon I\to\bC$ on a closed
subinterval $I$ of $\bR$. For a path $\gamma$ as in the
statement of the theorem, $\gamma(0)$ is necessarily
a root of $f$ and we have
$\max\{|f(\gamma(t))|\,;\,0\le t\le 1\}=|f(\beta)|$.
We will see that, in fact, $\gamma$ is a path of steepest
ascent for $|f|$.

For the proof, we consider the polynomial $f$ as a covering
of Riemann surfaces $f\colon\bC\to\bC$ of degree $N$,
ramified in a finite number of points.  Then any path
$\gamma\colon[0,1]\to\bC$ lifts into $N$ paths
$\gamma_1,\dots,\gamma_N\colon[0,1]\to\bC$ such that
$f^{-1}(\gamma(t))=\{\gamma_1(t),\dots,\gamma_N(t)\}$
for all $t\in[0,1]$.  The latter are not unique in
general, because of ramification, and are constructed by
pasting as in the proof of \cite[Theorem 4.14]{Fo1981}.
For a path $\gamma$ of the form $\gamma(t)=tf(\beta)$
with $f(\beta)\neq 0$, this leads to the following
statement.

\begin{lemma}
\label{descente:lemme:chemin}
Let $\beta\in\bC$ with $f(\beta)\neq 0$, and let $m=m(\beta)\ge 0$
denote the order of the derivative of $f$ at $\beta$.
Then, there exist $\delta\in(0,1)$ and $m+1$ paths
$\gamma_0,\dots,\gamma_m$ from $[0,1]$ to $\bC$ such that
\begin{itemize}
 \item[(i)] $\gamma_0(1)=\cdots=\gamma_m(1)=\beta$,
 \item[(ii)] $f(\gamma_0(t))=\cdots=f(\gamma_m(t))=tf(\beta)$
   for each $t\in[0,1]$,
 \item[(iii)] $\gamma_0(t),\dots,\gamma_m(t)$ are $m+1$
   distinct numbers for each $t\in(1-\delta,1)$.
\end{itemize}
Moreover, for each $j=0,1,\dots,m$ and each $t\in(0,1)$
such that $f'(\gamma_j(t))\neq 0$, the function $\gamma_j$ is
analytic at $t$ and its derivative $\gamma_j'(t)$ heads in
the direction where the norm $|f|$ of $f$ grows fastest.
\end{lemma}

The last assertion of the lemma means that $\gamma_0,\dots,
\gamma_m$ are paths of steepest ascent for the norm of $f$.
This is true in fact for any path $\gamma$ such that
$f(\gamma(t))=ct$ ($0\le t\le 1$) with a fixed
$c\in\bC\setminus\{0\}$ because the image of the map
$t\mapsto ct$ with $t\ge 0$ is a half line that is
orthogonal to the circles centered at the origin.
As the map $f\colon\bC\to\bC$ is conformal outside of the
ramification points, the preimage $\gamma$ of this curve
is orthogonal to the level curves of $|f|$
outside of these points. We will revisit the construction
of the paths $\gamma_j$ in Lemma \ref{arch:lemme:regions}.

\subsection*{Proof of Theorem \ref{descente:thm}}
If $f(\beta)=0$, the constant path $\gamma(t)=\beta$
for each $t\in [0,1]$ is the only possible choice and it
has the required properties.  Suppose from now on that
$f(\beta)\neq 0$.  Then the preceding lemma provides
a path $\gamma$ of the required type linking $\beta$
to a root of $f$.  Fix such a path.  For the computations,
we denote by $\alpha_1,\dots,\alpha_s$ the distinct
roots of $f$ in $\bC$ and by $n_1,\dots,n_s$ their
respective multiplicities so that
\begin{equation*}
 f(z)=(z-\alpha_1)^{n_1}\cdots(z-\alpha_s)^{n_s}.
 \label{descente:eq:f(z)}
\end{equation*}
We also denote by
$B$ the set of zeros of the derivative $f'$ of $f$.

By Gauss-Lucas theorem the set $B$ is contained in the
convex hull of the roots of $f$, thus $B\subset \cK$.  The fact
that the image of $\gamma$ is contained in $\cK$ admits a similar
proof. Indeed, suppose by contradiction that the image escapes
from $\cK$.  Then, since $\cK$ is convex, there exists a
half-plane containing $\cK$ but not the image of $\gamma$.
More precisely, there exist $a,b\in\bC$ with $|a|=1$ such that
$\Re(az+b)\le 0$ for each $z\in \cK$ and $\Re(a\gamma(t)+b)>0$
for at least one $t\in [0,1]$.  Choose $t_0\in[0,1]$
for which $\Re(a\gamma(t_0)+b)$ is maximal, and set
$z_0=\gamma(t_0)$.  Since $\Re(az_0+b)>0$, we have
$z_0\notin\cK$, thus $t_0\in (0,1)$ and $z_0\notin B$.
Therefore $\gamma$ is differentiable at $t_0$ with
$\Re(a\gamma'(t_0))=0$.  However, by differentiating
both sides of the equality $f(\gamma(t))=tf(\beta)$
at $t=t_0$, we obtain
\[
 a\gamma'(t_0)
  = \frac{af(\beta)}{f'(\gamma(t_0))}
  = \frac{af(z_0)}{t_0f'(z_0)}
  = \left(\sum_{\ell=1}^s\frac{t_0 n_\ell}{a(z_0-\alpha_\ell)}\right)^{-1}.
\]
As $\Re(a(z_0-\alpha_\ell))=\Re(az_0+b)-\Re(a\alpha_\ell+b)
\ge \Re(az_0+b) >0$ for $\ell=1,\dots,s$, we deduce that
$\Re(a\gamma'(t_0))>0$, a contradiction.

To estimate the length $L(\gamma)$ of $\gamma$, we use
the Cauchy-Crofton formula
\begin{align*}
 L(\gamma)
  =\frac{1}{4}
   \int_0^{2\pi}A(\theta)\,d\theta
 \quad
 &\text{where}\quad
 A(\theta)=\int_{-\infty}^\infty N(r,\theta)\,dr\,,\\
 &\text{and}\quad
 N(r,\theta)
  =\card \{t\in[0,1]\,;\, \Re(\gamma(t)e^{-i\theta})=r\}
\end{align*}
(see for example the beautiful proof of \cite{AD1997}).
Fix $r,\theta\in\bR$ and consider the polynomial
\[
 g_{r,\theta}(u)
  = \Im\left(\frac{f((r+iu)e^{i\theta})}{f(\beta)}\right)
  \in \bR[u].
\]
If $t_0\in[0,1]$ satisfies $\Re(\gamma(t_0)e^{-i\theta})=r$,
we may write $\gamma(t_0)=(r+iu_0)e^{i\theta}$ for some
$u_0\in\bR$.  Then we have $f((r+iu_0)e^{i\theta})=t_0f(\beta)$
and consequently $g_{r,\theta}(u_0)=0$.  As $\gamma$ is
injective on $[0,1]$ (because $f\circ\gamma$ is), this
means that $N(r,\theta)$ is at most equal to the number of
real roots of $g_{r,\theta}$.  But, as $f$ has degree $N$,
the polynomial $g_{r,\theta}(u)$ has degree at most $N$ and
its coefficient of $u^N$ is $\Im((ie^{i\theta})^N/f(\beta))$.
Thus, except possibly for the $2N$ values of $\theta\in[0,2\pi)$
for which this coefficient vanishes, we have $g_{r,\theta}\neq 0$
and thus $N(r,\theta)\le N$.

For fixed $\theta$, the set $\{\Re(ze^{-i\theta})\,;\, z\in D\}$
is an interval $I_\theta$ of $\bR$ of length $2R$.   As the image
of $\gamma$ is contained in $\cK\subset D$, we have
$N(r,\theta)=0$ if $r\notin I_\theta$.  We conclude
that $A(\theta) \le 2R N$ except for at most $2N$ values of
$\theta\in [0,2\pi)$, and thus $L(\gamma)\le \pi R N$.

%%%%%%%%%%%%%%%%%%%%%%%%%%%%%%%%%%%%%%%%%%%%%%%%%%%%%%%%%
%
%  A tree of paths between complex roots
%
%%%%%%%%%%%%%%%%%%%%%%%%%%%%%%%%%%%%%%%%%%%%%%%%%%%%%%%%%

%\newpage
\section{A tree of paths between complex roots}
\label{sec:arch}

As in the preceding section, we fix a non-constant
monic polynomial $f(z)\in\bC[z]$.  We denote by $N$ its
degree, by $A=\{\alpha_1,\dots,\alpha_s\}$ the set of its
complex roots, by $\cK$ the convex hull of $A$, and by
$R$ the radius of a closed disk $D$ containing $A$.
We also denote by $B=\{\beta_1,\dots,\beta_{p}\}$ the
set of roots of $f'(z)$ which are not roots of $f(z)$,
that is the set of zeros of the logarithmic derivative
$f'(z)/f(z)$.  Then we may write
\begin{align}
 f(z)&=(z-\alpha_1)^{n_1}\cdots(z-\alpha_s)^{n_s},
 \label{arch:eq:f(t)}\\
 f'(z)&=N(z-\alpha_1)^{n_1-1}\cdots(z-\alpha_s)^{n_s-1}
        (z-\beta_1)^{m_1}\cdots(z-\beta_p)^{m_p},
 \label{arch:eq:f'(t)}
\end{align}
for integers $n_1,\dots,n_s\ge 1$ with sum $N$,
and integers $m_1,\dots,m_p\ge 1$ with sum $s-1$.

For each $\beta\in\bC$, we denote by $m(\beta)$ the order
of $f'(z)$ at $\beta$.  With this
notation, we have $m_j=m(\beta_j)$ for $j=1,\dots,p$.
The goal of this section is to prove the following
result.

\begin{theorem}
\label{arch:thm:graphe}
There exists a tree $\cG$ with the following properties:
\begin{itemize}
 \item[(i)] Its set of vertices is $A$.
 \item[(ii)] It has $s-1$ edges, each one indexed by an
      element of $B$.
 \item[(iii)] For each $\beta\in B$, there are exactly
      $m(\beta)$ edges indexed by $\beta$.
 \item[(iv)] If $\{\alpha,\alpha'\}$ is an edge of\/ $G$ indexed
      by $\beta$, there exists a path $\gamma\colon [0,1]\to \bC$
      of length at most $2\pi R N$, contained in $\cK$, linking
      $\gamma(0)=\alpha$ to $\gamma(1)=\alpha'$, such that
     \[
      \gamma(1/2)=\beta
      \et
      \max_{0\le t\le 1} |f(\gamma(t)| = |f(\beta)|.
     \]
\end{itemize}
\end{theorem}

When all the roots of $f(z)$ are real, we have
$f(z)\in\bR[z]$ and we can give a very simple proof
of the theorem.  To this end, we may assume that the
roots are labelled in increasing order
$\alpha_1<\cdots<\alpha_s$.  Then, in each interval
$[\alpha_j,\alpha_{j+1}]$ with $1\le j\le s-1$, the function
$|f(z)|$ achieves its maximum in a zero $\beta_j$ of $f'(z)$
with $\alpha_j<\beta_j<\alpha_{j+1}$.  Since $B$ has
cardinality $p\le s-1$, this exhausts all the elements of $B$:
we have $p=s-1$ and $m_1=\cdots=m_{s-1}=1$.
We take for $\cG$ the graph with set of vertices $A$, whose edges
are the pairs $\{\alpha_j,\alpha_{j+1}\}$ indexed by
$\beta_j$ for $j=1,\dots,s-1$.  Then $\cG$ is a tree and,
for each $j=1,\dots,s-1$, the piecewise affine linear path
$\gamma_j$ with $\gamma_j(0)=\alpha_j$, $\gamma_j(1/2)=\beta_j$
and $\gamma_j(1)=\alpha_{j+1}$ fulfills the conditions
in (iv). Moreover its length is $\alpha_{j+1}-\alpha_j\le 2R$.

\subsection*{Step 1}
The proof of the general case requires several lemmas.
For each $\beta\in B$, we choose once for all
$m(\beta)+1$ paths $\gamma_{\beta,0},\dots,\gamma_{\beta,m(\beta)}$
with end point $\beta$ as in Lemma \ref{descente:lemme:chemin}.
Then we have $\gamma_{\beta,j}(0)\in A$ for $j=0,\dots,m(\beta)$.
Our goal is to show that these $m(\beta)+1$ points
of $A$ are distinct and that the graph $\cG$ with vertices
$\alpha_1,\dots,\alpha_s$ and edges
$\{\gamma_{\beta,0}(0),\gamma_{\beta,j}(0)\}$
with $\beta\in B$ and $1\le j \le m(\beta)$ satisfies the
properties (i) to (iv) from the theorem.  We start
with property (iv).

\begin{lemma}
\label{arch:lemme:longueur}
Let $\beta\in B$ and $j\in\{1,\dots,m(\beta)\}$.  Then
the path $\tgamma$ from $\gamma_{\beta,0}(0)$ to
$\gamma_{\beta,j}(0)$ given by
\[
 \tgamma(t)
  =\begin{cases}
   \gamma_{\beta,0}(2t) &\text{if $0\le t\le 1/2$,}\\
   \gamma_{\beta,j}(2-2t) &\text{if $1/2\le t\le 1$,}
   \end{cases}
\]
is contained in $\cK$, with length at most $2\pi R N$.
Moreover, it satisfies
\[
 \tgamma(1/2)=\beta
 \et
 \max_{0\le t\le 1} |f(\tgamma(t)| = |f(\beta)|.
\]
\end{lemma}

\begin{proof}
We have $B\subset \cK$ by Gauss-Lucas theorem.
Then, for each $\beta\in B$, Theorem \ref{descente:thm}
shows that the paths $\gamma_{\beta,0}$ and $\gamma_{\beta,j}$
are contained in $\cK$ with length at most $\pi RN$.  The
conclusion follows since these are path of steepest
ascent for $|f|$.
\end{proof}

\subsection*{Step 2}
We first prove the following result where $S=\bC\cup\{\infty\}$
stands for the Riemann sphere with its usual topology.
Afterwards, we use it to construct a tree $\cH$ on $A\cup B$.

\begin{lemma}
\label{arch:lemme:regions}
Let $\beta\in B$ and let $m=m(\beta)$.  There exist $\delta>0$
and $m+1$ continuous functions $\gamma^+_0,\dots,\gamma^+_m$
from $[1,\infty]$ to $S=\bC\cup\{\infty\}$ such that
\begin{itemize}
 \item[(i)] $\gamma^+_0(1)=\cdots=\gamma^+_m(1)=\beta$,
 \item[(ii)] $f(\gamma^+_0(t))=\cdots=f(\gamma^+_m(t))=tf(\beta)$
   for each $t\in[1,\infty]$,
 \item[(iii)] $\gamma^+_0(t),\dots,\gamma^+_m(t)$ are $m+1$
   distinct numbers for each $t\in(1,1+\delta)$.
\end{itemize}
Then, the curves $\Gamma^+_0=\gamma^+_0([1,\infty]),\dots,
\Gamma^+_m=\gamma^+_m([1,\infty])$ %pour $j=0,\dots,m$
meet only at the points $\beta$ and $\infty$ on $S$.
Moreover, their complement
$S\setminus(\Gamma^+_0\cup\cdots\cup\Gamma^+_m)$
is the union of $m+1$ disjoint connected open subsets
$\cR_0,\dots,\cR_m$ of $\bC$ such that
$\gamma_{\beta,j}([0,1))\subseteq \cR_j$ for $j=0,\dots,m$.
\end{lemma}

The proof is based on Jordan curve theorem and is illustrated
in Figure \ref{fig1}.

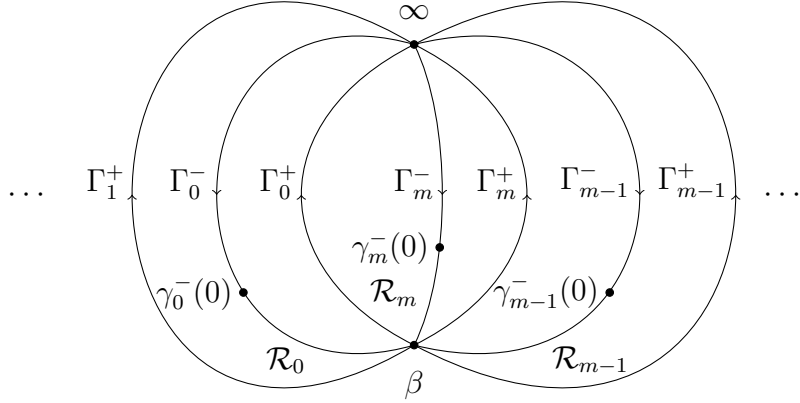
\begin{figure}[h]
 \begin{tikzpicture}
    \clip(-5.5,-3) rectangle (5.5,3);
    \draw[fill] (0, 2) circle (0.05);
    \node[above] at (0,2.2) {$\infty$};
    \draw[->-] (0,-2) .. controls (5.7,-5) and (5.7,5) .. (0,2); % 5.2
    \draw[->-] (0,2) .. controls (4,3) and (4,-3) .. (0,-2); % 3.75
    \draw[->-] (0,-2) .. controls (2,-1) and (2,1) .. (0,2);
    \draw[->-] (0,2) .. controls (0.5,1) and (0.5,-1) ..  (0,-2);
    \draw[->-] (0,-2) .. controls (-5,-5) and (-5,5) .. (0,2);
    \draw[->-] (0,2) .. controls (-3.5,3) and (-3.5,-3) .. (0,-2);
    \draw[->-] (0,-2) .. controls (-2,-1) and (-2,1) .. (0,2);
    \node[left] at (-4.7,0) {$\cdots$};
    \node[left, above] at (-4.1,-0.2) {$\Gamma^+_1$};
    \node[left, above] at (-3,-0.2) {$\Gamma^-_0$};
    \node[left, above] at (-1.8,-0.2) {$\Gamma^+_0$};
    \node[left, above] at (0,-0.2) {$\Gamma^-_m$};
    \node[left, above] at (1.1,-0.2) {$\Gamma^+_m$};
    \node[left, above] at (2.4,-0.2) {$\Gamma^-_{m-1}$}; % 2.3
    \node[left, above] at (3.7,-0.2) {$\Gamma^+_{m-1}$}; % 3.4
    \node[right] at (4.5,0) {$\cdots$};
    \draw[fill] (0,-2) circle (0.05);
    \node[below] at (0,-2.2) {$\beta$};
    \node[left] at (0.2,-1.3) {$\cR_m$};
    \node[left] at (3,-2.2) {$\cR_{m-1}$}; %2.65
    \node[left] at (-1.3,-2.2) {$\cR_{0}$};
    \draw[fill] (-2.27, -1.3) circle (0.05);
    \node[left] at (-2.27,-1.3) {$\gamma^-_0(0)$};
    \draw[fill] (0.34, -0.7) circle (0.05);
    \node[left] at (0.34,-0.7) {$\gamma^-_{m}(0)$};
    \draw[fill] (2.6, -1.3) circle (0.05); % 2.43
    \node[left] at (2.6,-1.3) {$\gamma^-_{m-1}(0)$}; % 2.43
 \end{tikzpicture}
\caption{Illustration for the proof of Lemma \ref{arch:lemme:regions}.}
\label{fig1}
\end{figure}

\begin{proof}
Upon putting $\ell=m+1$, we may write
$f(z)=f(\beta)(1+(z-\beta)^\ell g(z))$
where $g(z)$ is a polynomial with $g(\beta)\neq 0$.  Then, for
sufficiently small $\epsilon>0$, there exist an open
neighborhood $V$ of $\beta$ and a biholomorphic function
$h$ from $V$ to
$B(0,\epsilon)=\{z\in\bC\,;\,|z|<\epsilon\}$ satisfying
$h(\beta)=0$ and
\[
 f(z)=f(\beta)(1+h(z)^\ell)
\]
for each $z\in V$.  Fix such a choice of $\epsilon$, $V$ and $h$,
and set $\delta=\epsilon^\ell$ and $\rho=e^{\pi i/\ell}$.
For $j=0,\dots,m$, we define a continous function
$\gamma^+_j\colon [1,1+\delta)\to V$ by
\[
  \gamma^+_j(t)=h^{-1}\left(\rho^{2j}(t-1)^{1/\ell}\right)
  \quad
  (1\le t<1+\delta).
\]
Then, for fixed $t\in(1,1+\delta)$, the numbers
$z=\gamma^+_0(t),\dots,\gamma^+_m(t)$ are the $\ell$
distinct solutions of $f(z)=tf(\beta)$ with $z\in V$.
In particular, $\gamma^+_0,\dots,\gamma^+_m$ satisfy
Conditions (i) and (iii) of the lemma, as well as (ii)
for each $t\in[1,1+\delta)$.  For $j=0,\dots,m$, we extend
$\gamma^+_j$ to a continuous function
$\gamma^+_j\colon[1,\infty]\to S$ satisfying
$f(\gamma^+_j(t))=tf(\beta)$ for each $t\in[1,\infty]$.

Similarly, for $j=0,\dots,m$, we define a continuous
function $\gamma^-_j\colon (1-\delta,1]\to V$ by
\[
\gamma^-_j(t)=h^{-1}\left(\rho^{2j+1}(1-t)^{1/\ell}\right)
  \quad (1-\delta< t\le 1).
\]
For fixed $t\in(1-\delta,1)$, the numbers
$z=\gamma^-_0(t),\dots,\gamma^-_m(t)$ are the $\ell$
distinct solutions of $f(z)=tf(\beta)$ with $z\in V$, thus
they form a permutation of
$\gamma_{\beta,0}(t),\dots,\gamma_{\beta,m}(t)$.
This permutation being independent of $t$, there is no
loss of generality in assuming that $\gamma^-_j$ is the
restriction of $\gamma_{\beta,j}$ to $(1-\delta,1]$ for
$j=0,\dots,m$.  Then we extend each
$\gamma_{\beta,j}\colon[0,1]\to \bC$ to
a continuous function $\gamma^-_j\colon[-\infty,1]\to S$
such that $f(\gamma^-_j(t))=tf(\beta)$ for each $t\in[-\infty,1]$.

Put $\Gamma^-_j = \gamma^-_j([-\infty,1])$ and
$\Gamma^+_j = \gamma^+_j([1,\infty])$ for $j=0,\dots,m$,
and fix $j,k\in\{0,1,\dots,m\}$. The curves
$\Gamma^-_j$ and $\Gamma^+_k$ meet only at the points
$\beta$ and $\infty$ because if $\gamma^-_j(t)=\gamma^+_k(u)$
for some $t\in[-\infty,1]$ and $u\in[1,\infty]$, then
$tf(\beta) = uf(\beta)$, thus $t=u=1$ or $-t=u=\infty$.
Suppose now that $j<k$.  As the curves $\Gamma^+_j$ and
$\Gamma^+_k$ meet at infinity, there exists a smallest
$r\in[1+\delta,\infty]$ such that $\gamma^+_j(r)=\gamma^+_k(r)$.
For this choice of $r$, the union
$\gamma^+_j([1,r]) \cup \gamma^+_k([1,r])$ is a simple
closed curve $\Gamma$.  By Jordan curve theorem, its
complement in $S$ is thus the union of two 
connected open sets $\cR$ and $\cR'$ with 
boundary $\Gamma$.  On the other hand, we have
\[
 V\cap\Gamma
  = \gamma^+_j([1,1+\delta))\cup\gamma^+_k([1,1+\delta))
  = h^{-1}(P)
 \quad
 \text{where}
 \quad
 P=[0,\epsilon)\rho^{2j}\cup[0,\epsilon)\rho^{2k}.
\]
Moreover, $B(0,\epsilon)\setminus P$ is the union of two
disjoint connected open sets $\cU$ and $\cU'$ (open sectors
of the disk $B(0,\epsilon)$),  where $\cU$ contains the rays
$(0,\epsilon)\rho^{2i+1}$ with $j\le i<k$ and $\cU'$
those with $0\le i<j$ or $k\le i\le m$.
As $h\colon V\to B(0,\epsilon)$ is a homeomorphism,
$h^{-1}(\cU)$ and $h^{-1}(\cU')$ are disjoint connected open subsets
of $S$ whose union is $V\setminus\Gamma$.  We may assume
that $h^{-1}(\cU)\subset \cR$ and $h^{-1}(\cU')\subset \cR'$.  Then,
we obtain
\[
 \gamma^-_i((1-\delta,1))=h^{-1}((0,\epsilon)\rho^{2i+1})
 \subseteq
 \begin{cases}
   \cR  &\text{if $j\le i<k$,}\\
   \cR' &\text{else.}
 \end{cases}
\]
However, $\cR$ and $\cR'$ share the same boundary, contained
in $\Gamma^+_j \cup \Gamma^+_k$.  Thus none of the sets
$\Gamma^-_i \setminus \{\beta,\infty\}=\gamma_i^-((-\infty,1))$
meet this boundary.  As these are connected curves,
we conclude that $\Gamma^-_i \setminus \{\beta,\infty\}$ is
contained in $\cR$ if $j\le i<k$ and in $\cR'$ otherwise.
In particular, none of the open subsets $\cR$ and $\cR'$
of $\bC$ is bounded and consequently we must have $r=\infty$.
This means that $\Gamma^+_j$ and $\Gamma^+_k$ meet only at
$\beta$ and $\infty$.

With the above notation, we define $\cR_j=\cR$ for the choice
of $j\in\{0,\dots,m-1\}$ and $k=j+1$.  We also define
$\cR_m=\cR'$ for the choice of $j=0$ and $k=m$.  These are
connected open subsets of $\bC$ with $\gamma_{\beta,j}([0,1))
\subset \Gamma^-_j \setminus \{\beta,\infty\} \subset \cR_j$ for
$j=0,\dots,m$.  It remains to show that $\cR_0,\dots,\cR_m$
pairwise disjoint.  To this end, we first note that if
$j\neq k$, then $\cR_j \not\subseteq \cR_k$ since
$\Gamma^-_j\setminus\{\beta,\infty\}$ is contained in $\cR_j$
but not in $\cR_k$.  So if $\cR_j$ and $\cR_k$ intersect,
then $\cR_j$ meets the boundary of $\cR_k$.  Then $\cR_j$
contains at least one point
of $\Gamma^+_i\setminus\{\beta,\infty\}$
for some $i\in\{0,1,\dots,m\}$.  However, by the choice
of $\cR_j$, we have $\gamma^+_i(t)\notin\cR_j$ for each
$t\in(1,1+\delta)$.  Thus the curve
$\Gamma^+_i\setminus\{\beta,\infty\}$ is not fully contained
in $\cR_j$ and, as it is a connected set, it meets the
boundary of $\cR_j$ without being fully contained in it.
This is impossible because that boundary is the union of
two curves among $\Gamma^+_0,\dots,\Gamma^+_m$.
\end{proof}

\begin{lemma}
\label{arch:graphe:H}
For each $\beta\in B$, the $m(\beta)+1$ points
$\gamma_{\beta,j}(0)\in A$ with $0\le j\le m(\beta)$ are
distinct.  Moreover, let $\cH$ be the graph
whose set of vertices is $A\cup B$ and whose edges are
the pairs $\{\beta,\gamma_{\beta,j}(0)\}$
with $\beta\in B$ and $0\le j \le m(\beta)$.  Then $\cH$
is a tree.
\end{lemma}

\begin{proof}
The first assertion is a direct consequence of the preceding
lemma because, for $\beta\in B$ and $m=m(\beta)$, this lemma
provides disjoint connected open sets $\cR_0,\dots,\cR_m$
such that $\gamma_{\beta,j}(0)\in \cR_j$ for $j=0,\dots,m$.

Suppose that $\cH$ is not a forest.  Then $\cH$
contains a simple cycle: an elementary chain
$(a_1,\dots,a_k)$ with $k\ge 3$ such that $\{a_k,a_1\}$
is an edge of $\cH$.  Then, $k$ is an even integer and
the $a_i$'s belong alternatively to $A$ or $B$
according to the parity of $i$.  By permuting
cyclicly the elements of this chain if necessary, we may
assume that $a_1\in B$ and that $|f(a_1)|\ge |f(a_i)|$ for
$i=1,\dots,k$. Let $m=m(a_1)$ and let $\cR_0,\dots,\cR_m$
be the connected open sets associated to the point
$a_1\in B$ by Lemma \ref{arch:lemme:regions}.  For each
point $z\neq a_1$ outside of these open sets, we have
$f(z)=tf(a_1)$ for a real number $t>1$, thus $|f(z)|>|f(a_1)|$.
We set $a_{k+1}=a_1$ and, for $i=1,\dots,k$, we denote by
$\gamma_i$ the path of the form $\gamma_{\beta,j}$ which
links $a_i$ and $a_{i+1}$.  For each $t\in[0,1]$, we have
$f(\gamma_i(t))=t f(a_i)$ if $i$ is odd and
$f(\gamma_i(t)) = t f(a_{i+1})$ if $i$ is even.  In both
cases, this yields $|f(\gamma_i(t))| \le |f(a_1)|$, with
the strict inequality if $t\neq 1$.
As $a_1,\dots,a_k$ are distinct and as $\gamma_i(1) \in
\{a_3,\dots,a_{k-1}\}$ when $2\le i\le k-1$, we deduce
that the curve
\[
 \Gamma=\gamma_1([0,1))\cup
        \gamma_2([0,1])\cup\cdots\cup\gamma_{k-1}([0,1])
        \cup\gamma_k([0,1))
\]
is contained in $\cR_0\cup\cdots\cup\cR_m$.  As this is a
connected subset of $\bC$, it is therefore fully
contained in $\cR_j$ for some $j$.  Since $\gamma_1(1)=
\gamma_k(1)=a_1$, this implies that $\gamma_1=\gamma_k$,
thus $a_2=\gamma_1(0)=\gamma_k(0)=a_k$, which is impossible.

So $\cH$ is a forest.  Therefore, its number of connected
components is equal to its number of vertices minus its number of
edges, that is
\[
 |A\cup B| - \sum_{\beta\in B} (m(\beta)+1)
  = s - \sum_{\beta\in B} m(\beta) =1.
\]
Thus $\cH$ is in fact a tree.
\end{proof}

\subsection*{Step 4. Proof of Theorem \ref{arch:thm:graphe}}
Let $\cG$ be the graph whose set of vertices is $A$ and whose
edges are the pairs
\begin{equation}
 \label{etape5:eq}
 \{\gamma_{\beta,0}(0),\gamma_{\beta,j}(0)\}
 \quad
 \big(\beta\in B, \ 1\le j\le m(\beta)\big).
\end{equation}
Since $\cH$ is connected, so is the graph $\cG$. Since
$\cG$ possesses $s=|A|$ vertices and since
$\sum_{\beta\in B} m(\beta)=s-1$, we deduce that
the $s-1$ edges \eqref{etape5:eq} are distinct and
that $\cG$ is a tree. In particular, for each $\beta\in B$,
there are exactly $m(\beta)$ edges of $\cG$ indexed
by $\beta$ and Lemma \ref{arch:lemme:longueur} shows that,
for each of them, there exists a path satisfying
Condition (iv) of the theorem.

%%%%%%%%%%%%%%%%%%%%%%%%%%%%%%%%%%%%%%%%%%%%%%%%%%%%%%%%%%%%
%
%  Un calcul de semi-r\'esultant
%
%%%%%%%%%%%%%%%%%%%%%%%%%%%%%%%%%%%%%%%%%%%%%%%%%%%%%%%%%%%%

\section{Computation of a semi-resultant}
\label{sec:semi-resultant}

We first prove the following formula.

\begin{proposition}
\label{semi-resultant:prop}
With the notation of the preceding section, we have
\[
 N^N \prod_{j=1}^p f(\beta_j)^{m_j}
  = \prod_{i=1}^s
      \Big( n_i^{n_i}\prod_{k\neq i}(\alpha_i-\alpha_k)^{n_k}\Big).
\]
\end{proposition}

The left hand side of this equality is the semi-resultant of
$f(z)$ and $f'(z)$ in the sense of Chudnovsky \cite{Br1977,Ch1984}.

\begin{proof}
The formula for the derivative of a product applied to the
factorization \eqref{arch:eq:f(t)} of $f(z)$ yields
\[
 f'(z)=(z-\alpha_1)^{n_1-1}\cdots(z-\alpha_s)^{n_s-1}g(z)
\]
where
\[
 g(z)=\sum_{k=1}^s
      n_k(z-\alpha_1)\cdots\widehat{(z-\alpha_k)}\cdots(z-\alpha_s).
\]
By comparison with the factorization \eqref{arch:eq:f'(t)} of
$f'(z)$, we also find that
\[
 g(z)=N(z-\beta_1)^{m_1}\cdots(z-\beta_p)^{m_p}.
\]
Upon evaluating both expressions for $g(z)$ at $z=\alpha_k$, we
obtain
\[
 N\prod_{j=1}^p(\alpha_k-\beta_j)^{m_j}
  = n_k\prod_{i\neq k} (\alpha_k-\alpha_i)
 \quad
 (1\le k\le s).
\]
Since $m_1+\cdots+m_p=s-1$, these equalities may be rewritten as
\[
 N\prod_{j=1}^p(\beta_j-\alpha_k)^{m_j}
  = n_k\prod_{i\neq k} (\alpha_i-\alpha_k)
 \quad
 (1\le k\le s).
\]
As stated, this yields
\begin{align*}
 N^N \prod_{j=1}^p f(\beta_j)^{m_j}
   &= N^N \prod_{j=1}^p
       \Big( \prod_{k=1}^s(\beta_j-\alpha_k)^{n_k} \Big)^{m_j}\\
  &= \prod_{k=1}^s
       \Big( N \prod_{j=1}^p(\beta_j-\alpha_k)^{m_j} \Big)^{n_k}\\
  &= \prod_{k=1}^s
       \Big( n_k \prod_{i\neq k}(\alpha_i-\alpha_k) \Big)^{n_k}
   = \prod_{i=1}^s
      \Big( n_i^{n_i}\prod_{k\neq i}(\alpha_i-\alpha_k)^{n_k}\Big).
\qedhere
\end{align*}
\end{proof}

\begin{cor}
\label{semi-resultant:cor}
With the same notation, we have
\[
 N! \prod_{j=1}^p \big|f(\beta_j)\big|^{m_j}
  \le \prod_{i=1}^s
      \Big( n_i! \prod_{k\neq i}|\alpha_i-\alpha_k|^{n_k}\Big).
\]
\end{cor}

\begin{proof}
Since $N=n_1+\cdots+n_s$, we find
\[
 \frac{N!}{n_1!\cdots n_s!}
   \Big(\frac{n_1}{N}\Big)^{n_1}\cdots\Big(\frac{n_s}{N}\Big)^{n_s}
 \le \Big(\frac{n_1}{N}+\cdots+\frac{n_s}{N}\Big)^N
  = 1.
\]
This yields $N!\prod_{i=1}^s n_i^{n_i} \le N^N \prod_{i=1}^s n_i!$\,,
and the conclusion follows.
\end{proof}

%%%%%%%%%%%%%%%%%%%%%%%%%%%%%%%%%%%%%%%%%%%%%%%%%%%%%%%%%%%%
%
%  Volume of the Archimedean components
%
%%%%%%%%%%%%%%%%%%%%%%%%%%%%%%%%%%%%%%%%%%%%%%%%%%%%%%%%%%%%

\section{Volume of the Archimedean components}
\label{sec:volarch}

We are now ready to prove the upper bound estimate
in Theorem \ref{res:principal} (i).  The notation
is as in Section \ref{sec:resultat}.

\begin{theorem}
 \label{volarch:thm}
Let $v$ be an Archimedean place of $K$ and let $\cC_{\un,v}$ be
the convex body of $K_v^s$ defined in Section \ref{subsec:res}
for the choice of an $s$-tuple $\un=(n_1,\dots,n_s)\in\bN_+^s$.
Then, we have
\[
 \mu_v(\cC_{\un,v})^{1/d_v}\le c_v N^{2s-2} |\Delta_\un|_v
 \quad
 \text{with}
 \quad
 c_v = 2^s e^{sR_v} (2\pi R_v^s)^{s-1} |\Delta_\uun|_v^{-1},
\]
where $N=n_1+\cdots+n_s$, $R_v=\max_{1\le i<j\le s}|\alpha_i-\alpha_j|_v$,
and $\uun=(1,\dots,1)$.
\end{theorem}

\begin{proof}
To simplify, we may assume that $K\subset \bC$ and that
$|\alpha|_v=|\alpha|$ for each $\alpha\in K$. By permuting
$\alpha_1,\dots,\alpha_s$ if necessary, we may also assume that
$n_1\ge\cdots\ge n_s$ form a decreasing sequence.  We denote by $D$
the closed disk of radius $R_v$ and center
$(\alpha_1+\dots+\alpha_s)/s$ in $\bC$.  As this disk contains
$\alpha_1,\dots,\alpha_s$, it also contains the convex hull $\cK$
of these points.

Suppose first that $n_1\ge 2$ and let $r$ be the largest index
such that $n_r\ge 2$.  We form the polynomial
\[
 f(z)=\frac{f_\un(z)}{(z-\alpha_1)\cdots(z-\alpha_s)}
     =\prod_{i=1}^r (z-\alpha_j)^{n_i-1}.
\]
The set of its roots is $A=\{\alpha_1,\dots,\alpha_r\}$
and its degree is $N-s$.  Its derivative factors as
\[
 f'(z)=(N-s)(z-\alpha_1)^{n_1-2}\cdots(z-\alpha_r)^{n_r-2}
       (z-\beta_1)^{m_1}\cdots(z-\beta_p)^{m_p}
\]
where $B=\{\beta_1,\dots,\beta_p\}$ is the set of roots
of $f'(z)$ outside of $A$, and where $m_j$ is the multiplicity
of $\beta_j$ for $j=1,\dots,p$.  We choose a tree $G$
as in Theorem \ref{arch:thm:graphe} for this polynomial
$f(z)$.  By construction, the set of vertices of $G$ is $A$.
We now extend $G$ to a graph $\tG$ on
$\{\alpha_1,\dots,\alpha_s\}$ in the following way.
For each $j=r+1,\dots,s$, we choose a path
$\gamma_j\colon[0,1]\to\bC$ such that $\gamma_j(1)=\alpha_j$
and $f(\gamma_j(t))=tf(\alpha_j)$ as in Theorem
\ref{descente:thm}.  Then $\gamma_j(0)$ is a root of $f$,
thus an element of $A$, and we add the edge
$\{\gamma_j(0), \alpha_j\}$ to the graph $G$.
Finally, we choose $\alpha_1\in A$ as a root of
the resulting tree $\tG$.  Then,
$\cC_{\un,v}$ is contained in the set $\cK_v$ of all points
$(x_1,\dots,x_s)\in K_v^s$ satisfying
\[
 |x_1|_v
    \le e^{R_v} (N-1)!
\]
as well as
\[
 |x_ie^{\alpha_j-\alpha_i}-x_j|_v
 \le b_{i,j}
  := \max_{1\le k\le s}
      \left|
       \int_{\alpha_i}^{\alpha_j}
          f_{\un-\ue_k}(z)e^{\alpha_j-z}\dz
      \right|
\]
for each directed edge $(\alpha_i,\alpha_j)$ of $\tG$ with
$\alpha_i<\alpha_j$.  Since $\tG$ is a rooted tree,
Proposition \ref{graphes:prop} shows that the $s$
linear forms defining $\cK_v$ are linearly
independent, with determinant $\pm 1$.
Thus $\cK_v$ is a convex body of $K_v^s$ with
\begin{equation}
 \label{volarch:eq1}
 \mu_v(\cC_{\un,v})^{1/d_v}
   \le \mu_v(\cK_v)^{1/d_v}
   \le 2^s e^{R_v} (N-1)!
       \prod_{(\alpha_i,\alpha_j)\in E} b_{i,j}
\end{equation}
where $E$ stands for the set of directed edges of $\tG$.

For now, fix $(\alpha_i,\alpha_j)\in E$ and
$k\in\{1,\dots,s\}$.  By construction, we have $i\le r$,
that is $\alpha_i\in A$.  If $j\le r$, we also
have $\alpha_j\in A$, and $\{\alpha_i,\alpha_j\}$ is an
edge of $G$.  Then, Theorem \ref{arch:thm:graphe} associates
to this edge a point $\beta\in B$ and
a path $\gamma\colon [0,1]\to\bC$ of length at most
$2\pi R_v N$, contained in $\cK$, joining $\alpha_i$ and $\alpha_j$,
such that
\[
 \max_{0\le t\le 1}|f(\gamma(t))| = |f(\beta)|.
\]
This yields
\begin{align*}
 \left|
       \int_{\alpha_i}^{\alpha_j}
          f_{\un-\ue_k}(z)e^{\alpha_j-z}\dz
 \right|
 &=
 \left|
       \int_{\alpha_i}^{\alpha_j}
          f(z)(z-\alpha_1)\cdots\widehat{(z-\alpha_k)}\cdots(z-\alpha_s)
          e^{\alpha_j-z} \dz
 \right|\\
 &\le 2\pi R_v N |f(\beta)|\,
      \max_{z\in \cK}
        \big|
        (z-\alpha_1)\cdots\widehat{(z-\alpha_k)}\cdots(z-\alpha_s)
          e^{\alpha_j-z}
        \big| \\
 &\le 2\pi R_v^s e^{R_v} N |f(\beta)|\,,
\end{align*}
since $|z-\alpha_\ell|\le R_v$ for any $z\in\cK$ and $\ell=1,\dots,s$.
Finally, if $j>r$, we have $\alpha_i=\gamma_j(0)$ for the path $\gamma_j$
chosen earlier.  By Theorem \ref{descente:thm},
the image of $\gamma_j$ is contained in $\cK$, of length at most
$\pi R_v N\le 2\pi R_v N$.  Thus the same computation as above yields
\[
 \left|
       \int_{\alpha_i}^{\alpha_j}
          f_{\un-\ue_k}(z)e^{\alpha_j-z}\dz
 \right|
 \le 2\pi R_v^s e^{R_v} N |f(\alpha_j)|\,.
\]

Since each $\beta_j$ is associated to $m_j$ edges of $G$
and since $m_1+\cdots+m_p=r-1$, we deduce from \eqref{volarch:eq1}
that
\begin{equation}
 \label{volarch:eq2}
 \mu_v(\cC_{\un,v})^{1/d_v}
   \le 2^s e^{R_v} (N-1)! \big(2\pi R_v^s e^{R_v} N)^{s-1}
       \prod_{j=1}^p |f(\beta_j)|^{m_j}
       \prod_{j=r+1}^s |f(\alpha_j)|\,.
\end{equation}
As $n_k=1$ for $k>r$, Corollary \ref{semi-resultant:cor} gives
\[
 (N-s)! \prod_{j=1}^p \big|f(\beta_j)\big|^{m_j}
  \le \prod_{i=1}^r
      \Big( (n_i-1)! \prod_{k\neq i}|\alpha_i-\alpha_k|^{n_k-1}\Big).
\]
For $i=r+1,\dots,s$, we also find that
\[
 |f(\alpha_i)|=\prod_{k=1}^r|\alpha_i-\alpha_k|^{n_k-1}
   =(n_i-1)! \prod_{k\neq i} |\alpha_i-\alpha_k|^{n_k-1}.
\]
This implies that
\[
 (N-s)! \prod_{j=1}^p |f(\beta_j)|^{m_j}
       \prod_{j=r+1}^s |f(\alpha_j)|
 \le \prod_{i=1}^s
      \Big( (n_i-1)! \prod_{k\neq i}|\alpha_i-\alpha_k|^{n_k-1}\Big)
  = \frac{|\Delta_\un|_v}{|\Delta_\uun|_v}.
\]
Substituting this upper bound in \eqref{volarch:eq2},
we conclude that $\mu_v(\cC_{\un,v})^{1/d_v}\le c_vN^{2s-2}|\Delta_\un|_v$,
as in the statement of the theorem.
\end{proof}

%%%%%%%%%%%%%%%%%%%%%%%%%%%%%%%%%%%%%%%%%%%%%%%%%%%%%%%%%
%
%  A forest at ultrametric places
%
%%%%%%%%%%%%%%%%%%%%%%%%%%%%%%%%%%%%%%%%%%%%%%%%%%%%%%%%%

%\newpage
\section{A forest at ultrametric places}
\label{sec:ultra}

Let $v$ be an ultrametric place of $K$.  In this section we
use the terminology for graphs explained in section
\ref{sec:graphes} to build a rooted forest on an
arbitrary non-empty finite subset of $K_v$.  We start with
a preliminary construction.

\begin{proposition}
 \label{ultra:prop:arbre}
Let $A$ be a non-empty finite subset of $K_v$ and let $\alpha_0\in A$.
There exists a tree $\cG$ rooted in $\alpha_0$ having $A$
as its set of vertices, such that, for each
$\alpha,\beta,\gamma\in A$ with $\beta\in S_\cG(\alpha)$,
we have
\begin{equation}
 \label{ultra:prop:arbre:eq}
 \gamma\in D_\cG(\beta)
    \quad\Longleftrightarrow\quad
 |\alpha-\beta|_v >|\beta-\gamma|_v>0.
\end{equation}
\end{proposition}

\begin{proof}
We proceed by induction on the cardinality $|A|$ of $A$.
If $|A|=1$, there is nothing to prove. Suppose that
$|A|\ge 2$.  Let $\rho$ be the largest distance
between two elements de $A$, and let
$\{\alpha_0,\dots,\alpha_k\}$ be a maximal subset
of $A$ containing $\alpha_0$, whose elements
are at mutual distance $|\alpha_i-\alpha_j|_v=\rho$
for $0\le i<j\le k$.  Since $v$ is ultrametric, we
have $k\ge 1$ and the sets
\[
 A_i:=\{\beta\in A\,;\, |\alpha_i-\beta|_v<\rho\}
 \quad
 (0\le i \le k)
\]
form a partition of $A$.  For $i=0,\dots,k$, we have
$\alpha_i\in A_i$ and $|A_i|<|A|$, thus we may assume the
existence of a rooted tree $\cG_i=(\alpha_i,A_i,E_i)$
which fulfils Condition \eqref{ultra:prop:arbre:eq}
for each choice of $\alpha,\beta,\gamma \in A_i$ with
$\beta\in S_{\cG_i}(\alpha)$.  We set
\[
 E=E_0\cup\cdots\cup E_k
   \cup\{\{\alpha_0,\alpha_1\},\dots,\{\alpha_0,\alpha_k\}\}.
\]
Then $\cG=(\alpha_0,A,E)$ is a rooted tree.  Let
$\alpha,\beta,\gamma\in A$ with $\beta\in S_\cG(\alpha)$,
and let $i$ be the index for which $\alpha\in A_i$.
If $\beta\in A_i$, then $\beta\in S_{\cG_i}(\alpha)$
and $D_\cG(\beta)=D_{\cG_i}(\beta)$, thus
\[
 \gamma\in D_\cG(\beta)
    \quad\Longleftrightarrow\quad
 \gamma\in D_{\cG_i}(\beta)
    \quad\Longleftrightarrow\quad
 |\alpha-\beta|_v >|\beta-\gamma|_v>0.
\]
If instead $\beta\in A_j$ for some $j\neq i$, then
we must have $i=0$, $\alpha=\alpha_0$ and
$\beta=\alpha_j$.  Then $|\alpha-\beta|_v=\rho$ and
$D_\cG(\beta)=A_j\setminus\{\alpha_j\}$.  So we find
\[
 \gamma\in D_\cG(\beta)
    \quad\Longleftrightarrow\quad
 \rho>|\alpha_j-\gamma|_v>0
    \quad\Longleftrightarrow\quad
 |\alpha-\beta|_v >|\beta-\gamma|_v>0.
\]
Thus $\cG$ has the required property.
\end{proof}

As the proof shows, the graph $\cG$ constructed in this way
is not unique in general (since the choice
$\alpha_1,\dots,\alpha_k\in A$ is not unique).  This leads to
the following construction which in general is not unique
either.

\begin{theorem}
 \label{ultra:thm:foret}
Let $A$ be a non-empty finite subset of $K_v$, let $\delta>0$,
and let $R$ be a maximal subset of $A$ whose elements
are at mutual distance at least $\delta$.
Then, there exists a rooted forest $\cG$ having $A$ as
its set of vertices and $R$ as its set of roots, which
satisfies the following properties:
\begin{itemize}
\item[(i)] for any $\beta\in R$ and $\gamma\in A$, we have
\[
 \gamma\in D_\cG(\beta)
    \ssi
 \delta >|\beta-\gamma|_v>0;
\]
\item[(ii)] for any $\alpha,\beta,\gamma\in A$ with
$\beta\in S_\cG(\alpha)$, we have
\[
 \gamma\in D_\cG(\beta)
    \quad\Longleftrightarrow\quad
 |\alpha-\beta|_v >|\beta-\gamma|_v>0.
\]
\end{itemize}
\end{theorem}

\begin{proof}
For each $\rho\in R$, we define
\[
 A^{(\rho)}=\{\alpha\in A\,;\, |\alpha-\rho|_v<\delta\},
\]
and we choose a rooted tree $\cG^{(\rho)} =
(\rho,A^{(\rho)},E^{(\rho)})$ as in Proposition
\ref{ultra:prop:arbre}.  Since the sets $A^{(\rho)}$
with $\rho\in R$ form a partition of $A$, the union
of these graphs constitute a rooted forest $G=(R,A,E)$
where $E=\cup_{\rho\in R}E^{(\rho)}$.  By construction,
it satisfies Condition (i). To show that Condition (ii)
is also fulfilled, fix $\alpha,\beta,\gamma\in A$ with
$\beta\in S_\cG(\alpha)$, and let $\rho\in R$ such that
$\alpha\in A^{(\rho)}$.  Since $\beta\in S_\cG(\alpha)$,
we have $\beta\in A^{(\rho)}$ and
$D_\cG(\beta)=D_{\cG^{(\rho)}}(\beta)$.  Moreover, if
$\gamma$ satisfies $|\alpha-\beta|_v >|\beta-\gamma|_v$
then $|\beta-\gamma|_v<\delta$ and so $\gamma\in A^{(\rho)}$.
Thus Condition (ii) for $\alpha,\beta,\gamma$ is satisfied
in $\cG$ since it is satisfied in $G^{(\rho)}$.
\end{proof}

In terms of elementary chains, Conditions (i) et (ii)
of the theorem can be reformulated as follows: given
$\gamma\in A$, a sequence $(\gamma_1,\dots,\gamma_k)$
in $G$, with $k\ge 1$ and $\gamma_k\neq \gamma$, starting 
on a root $\gamma_1\in R$, can be extended to an 
elementary chain $(\gamma_1,\dots,\gamma_\ell)$ ending on
$\gamma_\ell=\gamma$ if and only if either we have $k=1$ and
$\delta >|\gamma_1-\gamma|_v > 0$ or the sequence
$(\gamma_1,\dots,\gamma_k)$ is an elementary chain with
$k\ge 2$ and $|\gamma_{k-1}-\gamma_k|_v > |\gamma_k-\gamma|_v>0$.

%%%%%%%%%%%%%%%%%%%%%%%%%%%%%%%%%%%%%%%%%%%%%%%%%%%%%%%%%%%%
%
%  Volume of the ultrametric components
%
%%%%%%%%%%%%%%%%%%%%%%%%%%%%%%%%%%%%%%%%%%%%%%%%%%%%%%%%%%%%

\section{Volume of the ultrametric components}
\label{sec:volultra}

We now complete the proof of Theorem \ref{res:principal}
by proving the remaining estimates in parts (ii) and (iii).  The
notation is as in Section \ref{sec:resultat}.

\begin{theorem}
Let $v$ be a place of $K$ above a prime number $p$,
let $\un=(n_1,\dots,n_s)\in\bN_+^s$ and let $N=n_1+\cdots+n_s$.
Then the sub-$\cO_v$-module $\cC_{\un,v}$
of $K_v^s$ defined in Section \ref{subsec:res}
satisfies
\[
 \mu_v(\cC_{\un,v})^{1/d_v}\le (p^3N)^s |\Delta_\un|_v\,.
\]
Moreover, if $|\alpha_i-\alpha_j|_v=1$
for each $i,j\in\{1,\dots,s\}$ with $i\neq j$, then we also have
\[
 \mu_v(\cC_{\un,v})^{1/d_v} = |\Delta_\un|_v.
\]
\end{theorem}

\begin{proof}
We apply Theorem \ref{ultra:thm:foret} to the set
$A=\{\alpha_1,\dots,\alpha_s\}$ with $\delta=p^{-1/(p-1)}$.
It provides a rooted forest $G$ with roots $R$, vertices $A$,
and edges $E$.
For each $\alpha\in A$, we define $x_\alpha=x_i$ and
$n_\alpha=n_i$ where $i$ is the index for which
$\alpha=\alpha_i$.  Then, $\cC_{\un,v}$ is
contained in the set $\cK_v$
of points $(x_1,\dots,x_s)\in K_v^s$ satisfying
\[
 |x_\beta|_v
    \le p^3 N \prod_{\gamma\in A}
              \max\{|\beta-\gamma|_v,\,\delta\}^{n_\gamma}
\]
for each root $\beta\in R$, as well as
\[
 |x_\alpha e^{\beta-\alpha} - x_\beta|_v
   \le p^3 N \prod_{\gamma\in A}
             \max\{|\alpha-\gamma|_v,\,|\beta-\gamma|_v\}^{n_\gamma}
\]
for each directed edge $(\alpha,\beta)\in E$ or equivalently
for each pair $\{\alpha,\beta\}$ with $\beta\in S_G(\alpha)$
(since we then have $|\beta-\alpha|_v<\delta$).  By Proposition
\ref{graphes:prop}, the above $s$ linear forms are
linearly independent, with determinant $\pm 1$. So $\cK_v$ is
a free sub-$\cO_v$-module of $K_v^s$ of rank $s$ with
\[
 \mu_v(\cC_{\un,v})^{1/d_v}
   \le \mu_v(\cK_v)^{1/d_v}
   \le (p^3N)^s
      \Delta' \Delta''
\]
where
\[
 \Delta'=\prod_{\substack{\beta\in R\\ \gamma\in A}}
         \max\{|\beta-\gamma|_v,\,\delta\}^{n_\gamma}
 \et
 \Delta''=\prod_{\substack{\alpha,\beta,\gamma\in A\\ \beta\in S_G(\alpha)}}
         \max\{|\alpha-\gamma|_v,\,|\beta-\gamma|_v\}^{n_\gamma}.
\]
Let $\beta,\gamma\in A$.  If $\beta\in R$,
Theorem \ref{ultra:thm:foret} (i) yields
\begin{equation}
 \label{volultra:eq:beta-gamma}
 \max\{|\beta-\gamma|_v,\,\delta\}
  =\begin{cases}
     \delta &\text{if $\gamma\in D_G(\beta)\cup\{\beta\}$,}\\
     |\beta-\gamma|_v &\text{else.}
   \end{cases}
\end{equation}
Otherwise, there exists a unique $\alpha\in A$
such that $\beta\in S_G(\alpha)$ and, since
\[
 |\alpha-\gamma|_v > |\beta-\gamma|_v
 \ssi
 |\alpha-\beta|_v > |\beta-\gamma|_v,
\]
Theorem \ref{ultra:thm:foret} (ii) yields
\begin{equation}
 \label{volultra:eq:alpha-beta-gamma}
 \max\{|\alpha-\gamma|_v,\,|\beta-\gamma|_v\}
  =\begin{cases}
     |\alpha-\gamma|_v &\text{if $\gamma\in D_G(\beta)\cup\{\beta\}$,}\\
     |\beta-\gamma|_v &\text{else.}
   \end{cases}
\end{equation}
Since $D_G(\beta)\cup\{\beta\}$ runs through all connected
components of $G$ as $\beta$ runs through $R$ and since
we have $\sum_{\gamma\in A} n_\gamma=N$, the equality
\eqref{volultra:eq:beta-gamma} implies that
\[
 \Delta'
 = \delta^N
   \prod_{\substack{\beta\in R\\
           \gamma\notin D_G(\beta)\cup\{\beta\}}}
         |\beta-\gamma|_v^{n_\gamma}.
\]
Furthermore, the equality \eqref{volultra:eq:alpha-beta-gamma}
implies that
\[
 \Delta''
 = \Bigg(\prod_{\substack{\alpha\in A\\ \gamma\in D_G(\alpha)}}
         |\alpha-\gamma|_v^{n_\gamma}\Bigg)
   \Bigg(\prod_{\substack{\beta\notin R\\ \gamma\notin D_G(\beta)\cup\{\beta\}}}
         |\beta-\gamma|_v^{n_\gamma}\Bigg)
\]
As a result we obtain
\[
 \Delta'\Delta''
 = \delta^N \prod_{\beta\in A} \prod_{\gamma\in A\setminus\{\beta\}}
    |\beta-\gamma|_v^{n_\gamma}.
\]
Since $\delta^N=\prod_{\beta\in A} \delta^{n_\beta}
\le \prod_{\beta\in A} |n_\beta!|_v
\le \prod_{\beta\in A} |(n_\beta-1)!|_v$, we conclude that
\[
 \mu_v(\cC_{\un,v})^{1/d_v}
   \le (p^3N)^s
       \prod_{\beta\in A}
         \Big|
         (n_\beta-1)!\prod_{\gamma\neq\beta} (\beta-\gamma)^{n_\gamma}
         \Big|_v
   = (p^3N)^s |\Delta_\un|_v.
\]
Finally, if $|\alpha_i-\alpha_j|_v=1$ for each
$i,j\in\{1,\dots,s\}$ with $i\neq j$, then $\cC_{\un,v}$ consists
of all points $(x_1,\dots,x_s)\in K_v^s$ satisfying
\[
 |x_i|_v \le |(n_i-1)!|_v
\]
for $i=1,\dots,s$, thus
\[
 \mu(\cC_{\un,v})^{1/d_v}=\prod_{i=1}^s |(n_i-1)!|_v = |\Delta_\un|_v.
 \qedhere
\]
\end{proof}

%%%%%%%%%%%%%%%%%%%%%%%%%%%%%%%%%%%%%%%%%%%%%%%%%%%%%%%%%%%%
%
%  A special case
%
%%%%%%%%%%%%%%%%%%%%%%%%%%%%%%%%%%%%%%%%%%%%%%%%%%%%%%%%%%%%

%\newpage

\section{A special case}
\label{sec:dexp}

The adelic convex bodies $\cC_\un$ associated to a point
$(\alpha_1,\dots,\alpha_s)\in K^s$ depend only on the
differences $\alpha_j-\alpha_i$ with $1\le i<j\le s$.  So,
we may always assume that $\alpha_1=0$.  Then for $s=2$,
we simply have a point $(0,\alpha)\in K^2$.  The proposition
below is an explicit form of Corollary \ref{res:cor}
for such a point and for diagonal pairs $\un=(n,n)\in\bN_+^2$.
In this statement, the adelic convex body is rescaled so that
its $v$-adic component is contained in $\cO_v^2$ for each
ultrametric place $v$ of $K$. We use it afterwards to prove Propositions
\ref{intro:prop:imaginaire} and \ref{intro:prop:e3} from the
introduction.  The notation is the same as in Section
\ref{sec:resultat}.

\begin{proposition}
\label{dexp:prop}
Let $\alpha\in K\setminus\{0\}$, and let $E$ be the finite
set of places $v$ of $K$ with $v\mid\infty$ or
$|\alpha|_v\neq 1$. For each place $v$ of $K$ with
$v\nmid\infty$, we set
$B_v=\min\big\{1,p^{1/(p-1)}|\alpha|_v\big\}$ where $p$ is
the prime number below $v$.  We also set
\[
 g=\sum_{v\in E} \frac{d_v}{d}
 \et
 B=\prod_{v\nmid \infty} B_v^{-d_v/d}.
\]
Finally, for each $n\in\bN_+$, we denote by $\tcC_n$
the adelic convex body of $K^2$ whose components $\tcC_{n,v}$
are defined as follows.
\begin{itemize}
\item[(i)] If $v\mid\infty$, then $\tcC_{n,v}$ is the set of
 points $(x,y)\in K_v^2$ such that
\[
 |x|_v \le n^{g-1}\frac{B^n(2n)!}{|\alpha|_v^n\,n!}
 \et
 |xe^\alpha-y|_v \le n^{g}\frac{B^n|\alpha|_v^n}{4^n\,n!}.
\]
\item[(ii)] If $v\mid p$ for a prime number $p$ and if
 $|\alpha|_v<p^{-1/(p-1)}$, then $\tcC_{n,v}$ consists of the
 points $(x,y)\in K_v^2$ such that
\[
 |x|_v \le 1
 \et
 |xe^\alpha-y|_v \le B_v^{2n}.
\]
\item[(iii)] If $v\mid p$ for a prime number $p$ and if
 $|\alpha|_v\ge p^{-1/(p-1)}$, then $\tcC_{n,v}=\cO_v^2$.
\end{itemize}
Then we have
\begin{equation}
\label{dexp:prop:eq1}
 c_4 n^{-2g+1} \le \lambda_1(\tcC_n)\le \lambda_2(\tcC_n) \le c_3
\end{equation}
for constants $c_3,c_4>0$ that depend only on $\alpha$ and $K$.
\end{proposition}

\begin{proof}
Let $n\in\bN_+$. We consider the adelic convex body
$\cC_\un$ constructed in Section \ref{subsec:res} for the
choice of $\alpha_1=0$, $\alpha_2=\alpha$ and $\un=(n,n)$.  For an
Archimedean place $v$ of $K$ associated to an embedding
$\sigma\colon K\hookrightarrow\bC$ and for $k=1,2$, we find
\begin{align*}
 \left|\int_0^{\sigma(\alpha)}
       f_{\un-\ue_k}^\sigma(z)e^{\sigma(\alpha)-z} dz \right|
 &\le |\sigma(\alpha)| e^{|\sigma(\alpha)|}
      \max_{t\in[0,1]}
        \left|f_{\un-\ue_k}^\sigma(\sigma(\alpha)t)\right|\\
 &\le e^{|\sigma(\alpha)|} |\sigma(\alpha)|^{2n}
      \max_{t\in[0,1]} t^{n-1}(1-t)^{n-1}\\
 &= 4 e^{|\alpha|_v} (|\alpha|_v/2)^{2n}.
\end{align*}
Thus the points $(x,y)$ of $\cC_{\un,v}$ satisfy
\[
 |x|_v\le e^{|\alpha|_v}(2n-1)!
 \et
 |xe^\alpha-y|_v\le 4e^{|\alpha|_v}(|\alpha|_v/2)^{2n}.
\]
This implies that $a_v\cC_{\un,v} \subseteq \tcC_{n,v}$
for
\[
 a_v=\frac{n^{g-1}B^n}{4e^{|\alpha|_v} \alpha^n (n-1)!}
    \in K_v^\times.
\]
For each prime number $p$ and each place $v$ of $K$
with $v\mid p$, we also find that
$a_v\cC_{\un,v} \subseteq \tcC_{n,v}$ for
\[
 a_v=\frac{p^{t_v}}{\alpha^n(n-1)!} \in K_v^\times
\]
where $t_v$ is the integer for which
\[
 2np^3B_v^{-n}\le p^{t_v} < 2np^4B_v^{-n}
\]
if $v\in E$, and $t_v=0$ otherwise.  This computation is
based simply on the fact that
$|(n-1)!|_v\ge |n!|_v\ge p^{-n/(p-1)}$.  Thus we
obtain $a\,\cC_\un\subseteq\tcC_n$ for the idele
$a=(a_v)_v \in K_\bA^\times$.

The product $\cD=\prod_v \{x\in K_v\,;\, |x|_v\le |a_v|_v\}
\subset K_\bA$
is an adelic convex body of $K$.  By the product formula
applied to the principal idele $\alpha^n(n-1)!\in K^\times$,
we find that the volume of $\cD$ is
\[
 \mu(\cD) = 2^{r_1}\pi^{r_2}\prod_v |a_v|_v^{d_v}
   = 2^{r_1}\pi^{r_2}
     \prod_{v\mid\infty}
       \Big(\frac{n^{g-1}B^n}{4e^{|\alpha|_v}}\Big)^{d_v}
     \prod_{p,\,v\mid p} p^{-t_vd_v}.
\]
Since $\prod_{v\mid\infty} B^{d_v} = B^d =
\prod_{v\nmid\infty}B_v^{-d_v}$, this can be rewritten as
\[
 \mu(\cD) = c_1 n^{d(g-1)} \prod_{p,\,v\mid p} (p^{t_v}B_v^n)^{-d_v},
\]
with $c_1=2^{r_1}\pi^{r_2} \prod_{v\mid\infty} (4e^{|\alpha|_v})^{-d_v}$.
Since $p^{t_v}B_v^n=1$ if $v\notin E$ and $p^{t_v}B_v^n<2np^4$ if
$v\in E$ and $v\mid p$, this yields
\[
 \mu(\cD)
   \ge c_2 n^{d(g-1)} \prod_{v\in E'} n^{-d_v}
   = c_2 n^{dg} \prod_{v\in E} n^{-d_v}
   = c_2,
\]
where $E'=\{v\in E\,;\,v\nmid\infty\}$ and
$c_2=c_1\prod_{v\in E'}(2p^4)^{-d_v}$.
By Theorem \ref{res:thm:MBV} (with $s=1$), we thus have
$\lambda_1(\cD)\le c_3$ where $c_3=(2^{r_1+r_2}|D(K)|^{1/2}c_2^{-1})^{1/d}$.
This means that there exists $\beta\in K^\times$ satisfying
$|\beta|_v\le c_3|a_v|_v$ for all Archimedean places $v$ of $K$
and $|\beta|_v\le |a_v|_v$ for all other places.  So, we obtain
\[
 \beta\cC_\un\subseteq c_3\tcC_n,
\]
which yields
\[
 \lambda_1(\tcC_n)\le \lambda_2(\tcC_n)\le c_3
\]
since $\beta\cC_\un$ contains the $K$-linearly independent points
$\beta\ua_{\un-\ue_1},\beta\ua_{\un-\ue_2}$ of $K^2$.
By Theorem \ref{res:thm:MBV} (with $s=2$), this implies that
\[
 \lambda_1(\tcC_n)
   \ge (2c_3)^{-1}\mu(c_3\tcC_n)^{-1/d}
     = (2c_3^2)^{-1}\mu(\tcC_n)^{-1/d}.
\]
Finally, for each place $v$ of $K$, we find that
\[
 \mu_v(\tcC_{n,v})^{1/d_v}
   \le \begin{cases}
         4n^{2g-1}B^{2n}
         &\text{if $v\mid\infty$,}\\
         B_v^{2n}
         &\text{else.}
        \end{cases}
\]
Since $B^d\prod_{v\nmid\infty}B_v^{d_v}=1$, this implies that
$\mu(\tcC_n)^{1/d}\le 4n^{2g-1}$, and so
\eqref{dexp:prop:eq1} follows with $c_4=(8c_3^2)^{-1}$.
\end{proof}

\begin{proof}[Proof of Proposition
\ref{intro:prop:imaginaire}]
Under the hypotheses of this proposition, the field
$K$ admits a single Archimedean place $\infty$,
induced by the inclusion $K\subset\bC$.  Moreover, in
the notation of Proposition \ref{dexp:prop}, the choice
of $\alpha$ leads to $B_v=1$ for any other place $v$ of
$K$. Thus, for each $n\in\bN_+$, we obtain
\[
 \tcC_n=\tcC_{n,\infty}\times\prod_{v\neq\infty}\cO_v^2,
\]
where $\tcC_{n,\infty}$ consists of all points $(x,y)$
of $K_\infty^2\subseteq\bC^2$ satisfying
\[
 |x| \le n^{g-1}\frac{(2n)!}{|\alpha|^n\,n!}
 \et
 |xe^\alpha-y| \le n^{g}\frac{|\alpha|^n}{4^n\,n!}.
\]
Moreover, by \eqref{dexp:prop:eq1}, we have
$\lambda_1(\tcC_n)\ge c_4n^{-2g+1}$ for a constant
$c_4>0$ depending only on $\alpha$ and $K$.

Let $(x,y)\in\cO_K^2$ with $x\neq 0$. The above implies that,
for each $n\in\bN_+$,
\[
 \text{if} \quad |x| < h(n) := c_4 n^{-g}\frac{(2n)!}{|\alpha|^n\,n!}
 \quad\text{then}\quad
 |xe^\alpha-y| \ge c_4 n^{-g+1}\frac{|\alpha|^n}{4^n\,n!} \,.
\]
If $|x|$ is large enough, we can find an integer $n\ge 2$ such that
$e^n\le h(n-1)\le |x|<h(n)$.  Then we have $n\le \log |x|$
and we obtain
\begin{align*}
 |x|\,|xe^\alpha-y|
   &\ge h(n-1)c_4 n^{-g+1}\frac{|\alpha|^n}{4^n\,n!} \\
    %&\ge  c_4^2 n^{-2g+1}\frac{(2n-2)!|\alpha|}{4^n\,n! (n-1)!}
    &\ge c_4^2|\alpha| n^{-2g}\binom{2n-2}{n-1}4^{-n}
      \ge c_5 n^{-2g-1}
      \ge c_5 (\log |x|)^{-2g-1},
\end{align*}
with $c_5=c_4^2|\alpha|/8$.  Since $\cO_K$ is a discret subset
of $\bC$, this leaves out a finite number of values of $x$.
To include them in the final lower bound, it suffices to
replace $c_5$ by a sufficiently small constant $c>0$.
\end{proof}

\begin{proof}[Proof of Proposition \ref{intro:prop:e3}]
We apply Proposition \ref{dexp:prop} with $K=\bQ$ and
$\alpha=3$.  In this context, we have $g=2$ and $B=B_3^{-1}=3^{1/2}$.
For a given $n\in\bN_+$, a simple computation shows that the
Archimedean component $\tcC_{n,\infty}$ of the adelic
convex body $\tcC_n$ satisfies
\begin{equation}
 \label{dexp:e3:eq1}
 n\cC_n \subseteq \tcC_{n,\infty}\subseteq n^2\cC_n,
\end{equation}
where $\cC_n$ is the convex body of $\bR^2$ defined in
Proposition \ref{intro:prop:e3}.  For its ultrametric
components, we find that
\[
 \tcC_{n,3}=\{(x,y)\in\bZ_3\,;\, |xe^3-y|_3\le 3^{-n}\}
\]
and $\tcC_{n,p}=\bZ_p^2$ for each prime number $p\neq 3$.
Thus the points of $\bQ^2$ which belong to the latter
components are exactly those of the lattice $\Lambda_n$ in
Proposition \ref{intro:prop:e3}.  Therefore, the minima
of $\tcC_n$ with respect to $\bQ^2$ in the adelic sense
are also the minima of $\tcC_{n,\infty}$ with respect to $\Lambda_n$
in the classical sense.  In view of the inclusions \eqref{dexp:e3:eq1},
this implies that $c_4n^{-2}\le \lambda_1(\cC_n,\Lambda)\le
\lambda_2(\cC_n,\Lambda)\le c_3n^2$ for the constants $c_3$ and $c_4$
given by Proposition \ref{dexp:prop}.
\end{proof}

%%%%%%%%%%%%%%%%%%%%%%%%%%%%%%%%%%%%%%%%%%%%%%%%%
%
%  Numerical computations
%
%%%%%%%%%%%%%%%%%%%%%%%%%%%%%%%%%%%%%%%%%%%%%%%%%

\section{Numerical computations}
\label{sec:num}

The formulas in
Appendix \ref{sec:rel} allow us to compute recursively
the diagonal Hermite approximations to $(1,e^3)$.
In this last section,
we explain how they can be used to compute efficiently the
partial quotients in the continued fraction expansion
of $e^3\in\bR$, and then to verify the inequalities
\eqref{intro:eq:loglog} from the introduction.
Our reference for continued fractions is \cite[Ch.~I]{Sc1980}.

Let $e^3=[a_0,a_1,a_2,\dots]$ denote the continued fraction
expansion of $e^3$.  Its first terms are
\[
 e^3=[20, 11, 1, 2, 4, 3, 1, 5, 1, 2, 16,\dots ],
\]
without any noticeable regularity.  For each integer $n\ge 0$,
we form the $n$-th convergent of $e^3$
\[
 \frac{p_n}{q_n}=[a_0,a_1,\dots,a_n]
\]
with $p_n\in\bZ$, $q_n\in\bN_+$ and $\gcd(p_n,q_n)=1$.
The table below lists all integers $n\ge 1$ with
$q_{n-1}\le 10^{500\,000}$ for which
\[
 a_n=\max\{a_1,a_2,\dots,a_n\}.
\]
For each of those integers, it provides the corresponding value
of $a_n$ as well as the value of $\log(q_{n-1})$ truncated at
the first decimal place.

\[
\begin{array}{r|ccccccccccc}
 n&1&10&31&87&133&211&244&388&2708&8055\\[2pt]
 \hline\\[-10pt]
 a_n&11&16&68&189&492&739&2566&5885&6384&10409\\[2pt]
 \log(q_{n-1})& 0.0& 9.4& 34.5& 97.9& 151.1& 256.6& 297.6&
 475.0& 3307.2& 9614.8
 \end{array}
 \quad
\]

\[
\begin{array}{r|ccccccc}
 n&9437&29508&30939&43482&91737&196440&476544\\[2pt]
 \hline\\[-10pt]
 a_n&19362&21981&46602&51140&315466&546341&569869\\[2pt]
 \log(q_{n-1})&11258.4& 34996.8& 36750.6& 51515.4& 109063.1& 233261.9& 566111.1
\end{array}
\]

To show how this implies the estimations \eqref{intro:eq:loglog},
define $\psi(x)=3\log(x)\log(\log(x))$ for each $x\ge e$.
For each pair $(p,q)\in\bZ^2$ with $q\ge 1$, there exists an
integer $n\ge 1$ such that $q_{n-1} \le q < q_n$.  By a
theorem of Lagrange \cite[Chapter~I, Theorem~5E]{Sc1980}, we have
\[
 |qe^3-p|
   \ge |q_{n-1}e^3-p_{n-1}|
   \ge \frac{1}{q_n+q_{n-1}}
   \ge \frac{1}{(a_n+2)q_{n-1}}.
\]
Assuming $q\ge 3$, this implies that
\begin{equation}
 \label{num:eq1}
 \psi(q)q\,|qe^3-p|
   \ge \frac{\psi(q_{n-1})}{a_n+2}.
\end{equation}
It is easy to check that the right hand side of
\eqref{num:eq1} is $\ge 1$ for all entries $n$ of the table
with $n\ge 10$.  Thus it is also $\ge 1$ for each integer
$n\ge 10$ with $q_{n-1}\le 10^{500\,000}$.  A quick
computation shows that this is also true
for $n=2,\dots,9$.  Thus the left hand side of
\eqref{num:eq1} is $\ge 1$ if $11\le q\le 10^{500\,000}$.
Finally, one checks that this is still true when
$4\le q\le 10$.

To compute the partial quotients $a_n$, put
\[
 C_n=\begin{pmatrix} 2n-4 &2n-1\\ 2n-1 &2n+2 \end{pmatrix}
 \et
 A_n=C_n\cdots C_1
\]
for each $n\ge \bN_+$.  By Corollary \ref{rel:cor2}
in the Appendix, the rows of $(n-1)!A_n$ are
Hermite's approximations $\ua_{n-1,n}$ and $\ua_{n,n-1}$
to $(1,e^3)$.  Thus we have
\begin{equation}
 \label{num:eq2}
 \lim_{n\to\infty} A_n\begin{pmatrix} e^3\\-1\end{pmatrix}
 = \begin{pmatrix} 0\\0\end{pmatrix}\,.
\end{equation}
We also note that, for each $n\ge 2$, the matrices
$C_n$ and $A_n$ belong to the set
\[
 \cM
  =\left\{
    \begin{pmatrix}t&u\\t'&u'\end{pmatrix}\in\Mat_{2\times 2}(\bZ)
    \,;\, 0\le t<u,\ 0\le t'<u' \et tu'\neq t'u
   \right\}.
\]
This is clear for the matrices $C_n$.  For the matrices $A_n$,
this follows from the fact that $\cM$ is closed under matrix
multiplication.

In general, if $A=\begin{pmatrix}t&u\\t'&u'\end{pmatrix}
\in\cM$, the ratios $t/u$ and $t'/u'$ admit unique
continued fraction expansions
\[
 \frac{t}{u}=[a_0,a_1,\dots,a_\ell]
 \et
 \frac{t'}{u'}=[a'_0,a'_1,\dots,a'_{\ell'}]
\]
with $a_0=a'_0=0$, $a_\ell\ge 2$ if $\ell\ge 1$, and
$a'_{\ell'}\ge 2$ if $\ell'\ge 1$.  Let
$(a_0,\dots,a_k)$ be the common initial part of the sequences
$(a_0,\dots,a_\ell)$ and $(a'_0,\dots,a'_{\ell'})$.  When $k=0$,
that is when $t=0$ or $t'=0$ or
$\lfloor u/t\rfloor\neq \lfloor u'/t'\rfloor$,
we say that $A$ is \emph{reduced}.  Then, we find that
\[
 A = R\begin{pmatrix}0&1\\1&a_k\end{pmatrix}
     \cdots\begin{pmatrix}0&1\\1&a_1\end{pmatrix}
\]
where $R\in\cM$ is reduced, with the convention that the
right hand side is $R$ when $k=0$.
In particular, for each $n\ge 2$, we obtain
\[
 A_n=R_n\begin{pmatrix}0&1\\1&a_{k(n)}\end{pmatrix}
     \cdots\begin{pmatrix}0&1\\1&a_1\end{pmatrix}
\]
for a reduced matrix $R_n\in\cM$, integers
$0\le k(1)\le k(2)\le \cdots$ and positive integers
$a_1,a_2,\dots$ such that
\begin{equation}
 \label{rel:eq:recurrence}
 C_{n+1}R_n=R_{n+1}\begin{pmatrix}0&1\\1&a_{k(n+1)}\end{pmatrix}
     \cdots\begin{pmatrix}0&1\\1&a_{k(n)+1}\end{pmatrix},
\end{equation}
with the convention that the product on the right is $R_{n+1}$
when $k(n+1)=k(n)$.  By \eqref{num:eq2}, the integers
$k(n)$ go to infinity with $n$ and so we conclude that
\[
 e^{-3}=[0,a_1,a_2,\dots]
 \et
 e^3=[a_1,a_2,\dots]
\]
are the respective continued fraction expansions of
$e^{-3}$ and $e^3$.  Therefore, to compute their partial
quotients $a_k$, it suffices to compute recursively the matrices
$R_n$ whose coefficients are in practice much smaller then
those of $A_n$ (we may also at each step factor out the
power of $3$ dividing $R_n$).  To further save
computation time we do not compute exactly the integers
$q_n$ but keep only a floating point approximation
of them (in practice we use 10 significative decimal digits).
In this way, it takes slightly above an hour of CPU time
to produce the tables using MAPLE software with a 64 bits
intel i5 processor.

%%%%%%%%%%%%%%%%%%%%%%%%%%%%%%%%%%%%%%%%%%%%%%%%%%%%%%
%
%  Recurrence relations
%
%%%%%%%%%%%%%%%%%%%%%%%%%%%%%%%%%%%%%%%%%%%%%%%%%%%%%%

\appendix

\section{Recurrence relations}
\label{sec:rel}

The notation being as in Section \ref{subsec:approx:Hermite}
we extend the definition of $f_\un(z)$, $P_\un(z)$ and
$a_\un$ to any $s$-tuple $\un\in\bZ^s$ by setting
\[
 f_\un(z)=P_\un(z)=0
 \et
 a_\un=(0,\dots,0)
 \quad\text{if}\quad
 \un\notin\bN^s.
\]
For each $\un\in\bN_+^s$, we denote by $A_\un$ the matrix whose
$\ell$-th row is $\ua_{\un-\ue_\ell}$ for $\ell=1,\dots,s$.
In \cite[\S\S IX-X]{He1873}, Hermite provides a
recurrence formula linking $A_{\un+\uun}$ to $A_\un$ where 
$\uun=(1,\dots,1)$.  Here we give more general recurrence
relations based on the same principle.  The formula
\eqref{rel:eq1} below is due to Hermite
\cite[\S IX, p.~230]{He1873} when $\un\in\bN_+^s$.

\begin{proposition}
\label{rel:prop}
Let $\un=(n_1,\dots,n_s)\in\bN^s$. We have
\begin{equation}
 \label{rel:eq1}
 a_\un = (f_\un(\alpha_1),\dots,f_\un(\alpha_s))
          + \sum_{j=1}^s n_ja_{\un-\ue_j}\,.\\
\end{equation}
Moreover, if $k,\ell\in\{1,\dots,s\}$ with $n_k\ge 1$,
we also have
\begin{equation}
 \label{rel:eq2}
 a_{\un+\ue_\ell-\ue_k}
   = a_\un + (\alpha_k-\alpha_\ell)a_{\un-\ue_k}.
\end{equation}
\end{proposition}

\begin{proof}
Leibniz formula for the derivative of a product gives
\[
 f_\un'(z)=\sum_{j=1}^s n_j f_{\un-\ue_j}(z).
\]
Taking the sum of all derivatives on both sides of
this equality, we obtain
\[
 P_\un(z)=f_\un(z)+\sum_{j=1}^s n_j P_{\un-\ue_j}(z)
\]
and \eqref{rel:eq1} follows.  The formula \eqref{rel:eq2}
is trivial if $k=\ell$.
Suppose that $k\neq \ell$ and $n_k\ge 1$ so that
$\un-\ue_k\in\bN^s$. Then we find
\[
 f_{\un+\ue_\ell-\ue_k}(z) - f_\un(z)
   = (z-\alpha_\ell)f_{\un-\ue_k}(z)-(z-\alpha_k)f_{\un-\ue_k}(z)
   = (\alpha_k-\alpha_\ell)f_{\un-\ue_k}(z).
\]
Taking again the sum of the derivatives, this yields
\[
 P_{\un+\ue_\ell-\ue_k}(z)
   = P_\un(z) + (\alpha_k-\alpha_\ell)P_{\un-\ue_k}(z)
\]
and \eqref{rel:eq2} follows.
\end{proof}

\begin{cor}
\label{rel:cor1}
Let $\un=(n_1,\dots,n_s)\in \bN_+^s$ and $\ell\in\{1,\dots,s\}$.
Then we have
\[
 A_{\un+\ue_\ell} = M_{\un,\ell} A_\un
\]
where
\[
 M_{\un,\ell}
   =\begin{pmatrix}
     n_1+(\alpha_1-\alpha_\ell) &n_2 &\cdots &n_s\\
     n_1 &n_2+(\alpha_2-\alpha_\ell) &\cdots &n_s\\
     \vdots &\vdots&\ddots&\vdots\\
     n_1 &n_2 &\cdots &n_s+(\alpha_s-\alpha_\ell)
    \end{pmatrix}.
\]
\end{cor}

\begin{proof}
As the entries of $\un$ are positive, the polynomial $f_\un$
vanishes at all points $\alpha_1,\dots,\alpha_s$ and the
formulas of Proposition \ref{rel:prop} yield
\[
 a_{\un+\ue_\ell-\ue_k}
   = (\alpha_k-\alpha_\ell)a_{\un-\ue_k} + \sum_{j=1}^s n_j a_{\un-\ue_j}
 \quad
 (1\le k\le s).
\qedhere
\]
\end{proof}

When $s=2$, this provides a quick way of computing
the matrices $A_{n,n}$.

\begin{cor}
\label{rel:cor2}
Suppose that $s=2$, $\alpha_1=0$ and $\alpha_2=\alpha\in K\setminus\{0\}$.
Then, for each $n\in\bN_+$, we have
\begin{equation}
 \label{rel:cor2:eq}
 A_{n,n}
  = \begin{pmatrix}
      P_{n-1,n}(0) &P_{n-1,n}(\alpha)\\
      P_{n,n-1}(0) &P_{n,n-1}(\alpha)
    \end{pmatrix}
  = (n-1)! C_nC_{n-1}\cdots C_1
\end{equation}
where
\[
 C_i
   =\begin{pmatrix}
     2i-1-\alpha &2i-1\\
     2i-1 &2i-1+\alpha
    \end{pmatrix}
 \quad
 (i\in\bN_+).
\]
\end{cor}

\begin{proof}
We find that $P_{0,1}(z)=z+1-\alpha$ and $P_{1,0}(z)=z+1$,
thus $A_{1,1}=C_1$.  In general, for an integer $n\ge 1$,
the formulas of the preceding corollary give
\[
 A_{n+1,n+1}
  = \begin{pmatrix}
     n &n+1\\
     n &n+1+\alpha
    \end{pmatrix}
    \begin{pmatrix}
     n-\alpha &n\\
     n &n
    \end{pmatrix}
    A_{n,n}
  = n C_{n+1} A_{n,n}
\]
and the conclusion follows by induction on $n$.
\end{proof}

\vfill
 \vfill

\small
\vbox{
\hbox{Damien \sc Roy}\par
\hbox{D\'epartement de math\'ematiques et de statistique}\par
\hbox{Universit\'e d'Ottawa}\par
\hbox{150 Louis Pasteur}\par
\hbox{Ottawa, Ontario}\par
\hbox{Canada K1N 6N5}
}

\end{document}